\documentclass{amsart}

\usepackage{amssymb}
\usepackage{amsmath}
\usepackage{amscd}
\usepackage{ascmac}
\usepackage{graphicx}
\usepackage{pinlabel}

\theoremstyle{plain}
\newtheorem{thm}{Theorem}[section]

\newtheorem{lem}[thm]{Lemma}
\newtheorem{prop}[thm]{Proposition}

\theoremstyle{definition}
\newtheorem{defn}[thm]{Definition}
\newtheorem{exmp}[thm]{Example}

\theoremstyle{remark}
\newtheorem{rem}[thm]{Remark}

\def\co{\colon\thinspace}

\begin{document}

\title[Counting Dirac braid relators]% end with percent
 {Counting Dirac braid relators and hyperelliptic Lefschetz fibrations} 
\author[H.~Endo]{Hisaaki Endo}
\address{Department of Mathematics\\Tokyo Institute of Technology\\
2-12-1 Oh-okayama\\Meguro-ku\\Tokyo 152-8551\\Japan}
\email{endo@math.titech.ac.jp}
\author[S.~Kamada]{Seiichi Kamada}
\address{Department of Mathematics\\Osaka City University\\
3-3-138 Sugimoto\\Sumiyoshi-ku\\Osaka 558-8585\\Japan}
\email{skamada@sci.osaka-cu.ac.jp}
\date{August 3, 2015; MSC 2000: primary 57M15, secondary 57N13}

\maketitle

\begin{abstract}
We define an invariant $w$ for hyperelliptic Lefschetz fibrations over closed oriented surfaces, 
which counts the number of Dirac braids included intrinsically in the monodromy, 
by using chart description introduced by the second author. 
As an application, we prove that two hyperelliptic Lefschetz fibrations of genus $g$ 
over a given base space 
are stably isomorphic if and only if they have the same numbers of singular fibers of each type 
and they have the same value of $w$ if $g$ is odd. 
We also give examples of pair of hyperelliptic Lefschetz fibrations with the same numbers of 
singular fibers of each type which are not stably isomorphic. 
\end{abstract}

%%%%%%%%%%%%% Section 1 %%%%%%%%%%%%%%%

\section{Introduction}

%%%%%%%%%%%%%%%%%%%%%%%%%%%%%%%%%

Since the seminal works of Donaldson \cite{Donaldson1999} 
and Gompf \cite{GS1999} in 1998, 
Lefschetz fibrations have been investigated from various viewpoints. 
%Although the classification of isomorphism classes of Lefschetz fibrations is a matter of great 
%importance, 
%it is a hard problem to deal with because there exists a Lefschetz fibration with prescribed 
%fundamental group (see \cite{Gompf1995} and \cite{Donaldson1999}). 
Several kinds of equivalence classes of a Lefschetz fibration 
such as isomorphism class, diffeomorphism type, homeomorphism type, 
and homotopy type, 
and their relations have been studied by many authors. 
The above work of Donaldson together with a theorem of Gompf  \cite{Gompf1995} 
implies that there exists a Lefschetz fibration with prescribed fundamental group, 
which means the classification of all homotopy types of Lefschetz fibrations is not possible 
in principle. 
Baykur \cite{Baykur2016} proved that any symplectic $4$-manifold which is not a rational 
or ruled surface, after sufficiently many blow-ups, admits an arbitrary number of 
nonisomorphic Lefschetz fibrations of the same genus 
(see also Park and Yun \cite{PY2009}, \cite{PY2015}).  
It clarified a significant difference between isomorphism classes and diffeomorphism types. 
Two Lefschetz fibrations of the same genus over a given base space 
are called stably isomorphic if they become isomorphic after fiber-summed with 
the same number of copies of a `universal' Lefschetz fibration. 
Auroux \cite{Auroux2005} obtained a sufficient condition for two Lefschetz fibrations 
over the $2$-sphere each of which admits a section to be stably isomorphic. 
Hasegawa, Tanaka, and the authors \cite{EHKT2014} relaxed Auroux's condition to 
obtain a necessary and sufficient condition 
for two Lefschetz fibrations over a closed surface to be stably isomorphic. 
Thus the stable isomorphism problem, in contrast to the isomorphism one, 
turned out to be within reach. 

Hyperelliptic Lefschetz fibrations are Lefschetz fibrations for which the image of the 
monodromy is included in the hyperelliptic mapping class group. 
They are considered to be a natural generalization of elliptic surfaces 
because several properties are common to these two kinds of fibrations. 
For instance, many of fibrations can be obtained by branched covering construction, 
the signature of a fibration localizes on the singular fibers, 
typical fibrations are used as building blocks for constructions of more complicated fibrations 
and $4$--manifolds, etc (see Siebert and Tian \cite{ST1999}, 
Fuller \cite{Fuller2000}, Endo \cite{Endo2000}, and Endo and Nagami \cite{EN2004}). 
Although the classification of isomorphism classes of irreducible hyperelliptic Lefschetz 
fibrations was partially established in genus two case by Siebert and Tian \cite{ST2003}, 
it seems there is little prospect for a complete classification of isomorphism classes 
of hyperelliptic Lefschetz fibrations. 
In fact, there are infinitely many distinct Lefschetz fibrations of genus two with 
the same numbers of singular fibers of each type (see Baykur and Korkmaz \cite{BK2015}, 
cf. Ozbagci and Stipsicz \cite{OS2000} and Korkmaz \cite{Korkmaz2001}). 

%It is necessary for us to 
%develop constructions and invariants of hyperelliptic Lefschetz fibrations 
%for the purpose of clarifying them as well as elliptic surfaces. 

In the present paper, we define 
a $\mathbb{Z}_2$--valued invariant $w$ 
for hyperelliptic Lefschetz fibrations over closed oriented surfaces 
by using chart description introduced by the second author \cite{Kamada1992} 
(Definition \ref{w2} and Proposition \ref{inv}). This invariant 
counts the number of `Dirac braids' included intrinsically in the monodromy representation 
of a hyperelliptic Lefschetz fibration of genus $g$. 
The Dirac braid is a full twist on all strands in the $(2g+2)$-string braid group 
$B_{2g+2}(S^2)$ of a $2$-sphere, 
which corresponds to a Dehn twist around all marked points in the mapping class group 
$\mathcal{M}_{0,2g+2}$ of a $2$-sphere with $2g+2$ marked points, 
and to a maximal chain relator in the hyperelliptic mapping class group $\mathcal{H}_g$ of 
a connected closed oriented surface of genus $g$, 
under natural homomorphisms (see \S \ref{three} for details). 
The relator $r_4$ in $\mathcal{M}_{0,2g+2}$ corresponding to the Dirac braid 
is called the Dirac braid relator (see \S \ref{chartdes}). 
The proof of the invariance of $w$ under chart moves is the most technical part of this paper 
(Proposition \ref{typew} and Proposition \ref{transition}). 
We expect the invariant $w$ to coincide with a $\mathbb{Z}_2$-valued invariant 
for hyperelliptic Lefschetz fibrations of odd genus mentioned by Auroux and Smith 
\cite{AS2008} (Remark \ref{as}). 
Employing the invariant $w$, 
we prove that two hyperelliptic Lefschetz fibrations of genus $g$ 
over a given base space 
are stably isomorphic if and only if they have the same numbers of singular fibers of each type 
and they have the same value of $w$ if $g$ is odd (Theorem \ref{stable}). 
Two hyperelliptic Lefschetz fibrations of genus $g$ over a given base space 
are called stably isomorphic if they become isomorphic after fiber-summed with 
the same number of copies of a hyperelliptic Lefschetz fibration on 
$\mathbb{CP}^2\# (4g+5)\overline{\mathbb{CP}}^2$, 
which is a natural generalization of the rational elliptic surface $E(1)$ (Definition \ref{stiso}). 
We also give examples of pair of hyperelliptic Lefschetz fibrations with the same numbers of 
singular fibers of each type which are not stably isomorphic 
(Example \ref{not1} and Example \ref{not2}).

%%%%%%%%%%%%% Section 2 %%%%%%%%%%%%%%%

\section{Lefschetz fibrations and hyperelliptic structures}

%%%%%%%%%%%%%%%%%%%%%%%%%%%%%%%%%

In this section we review a precise definition and basic properties of Lefschetz fibrations 
and introduce a notion of hyperellipticity for Lefschetz fibrations. 
See also Matsumoto \cite{Matsumoto1996}, Gompf and Stipsicz \cite{GS1999}, 
and Endo and Kamada \cite{EK2013}. 

\subsection{Lefschetz fibrations and their monodromies}

We begin with a precise definition of Lefschetz fibration. 
Let $\Sigma_g$ be a connected closed oriented surface of genus $g$. 

\begin{defn}\label{LF}
Let $M$ and $B$ be connected closed oriented smooth $4$--manifold and $2$--manifold, 
respectively. 
A smooth map $f\co M\rightarrow B$ is called a(n achiral) {\it Lefschetz fibration} of 
genus $g$ if it satisfies the following conditions: 

(i) the set $\Delta\subset B$ of critical values of $f$ is finite 
and $f$ is a smooth fiber bundle 
over $B-\Delta$ with fiber $\Sigma_g$; 

(ii) for each $b\in\Delta$, there exists a unique 
critical point $p$ in the {\it singular fiber} $F_b:=f^{-1}(b)$ 
such that $f$ is locally written as 
$f(z_1,z_2)=z_1z_2$ or $\bar{z}_1z_2$ with respect to some local complex 
coordinates around $p$ and $b$ which are compatible with 
orientations of $M$ and $B$; 

(iii) no fiber contains a $(\pm 1)$--sphere. 

We call $M$ the {\it total space}, $B$ the {\it base space}, and $f$ the {\it projection}. 
We call $p$ a critical point of {\it positive type} (resp. of {\it negative type}) 
and $F_b$ a singular fiber of {\it positive type} (resp. of {\it negative type}) if $f$ is locally 
written as $f(z_1,z_2)=z_1z_2$ (resp. $f(z_1,z_2)=\bar{z}_1z_2$) in (ii). 
For a regular value $b\in B$ of $f$, $f^{-1}(b)$ is often called a {\it general fiber}. 
\end{defn}

\begin{defn} 
Let $f\co M\rightarrow B$ and $f'\co M'\rightarrow B$ be Lefschetz fibrations of genus $g$ 
over the same base space $B$. We say that $f$ is {\it isomorphic} to $f'$ 
if there exist orientation preserving diffeomorphisms $H\co M\rightarrow M'$ 
and $h\co B\rightarrow B$ which satisfy $f'\circ H=h\circ f$. 
If we can choose such an $h$ isotopic to the identity relative to a given base point 
$b_0\in B$, we say that $f$ is {\it strictly isomorphic} to $f'$. 
\end{defn}

Let $\mathcal{M}_g$ be the mapping class group of $\Sigma_g$, 
namely the group of all isotopy classes of orientation preserving 
diffeomorphisms of $\Sigma_g$. 
We assume that $\mathcal{M}_g$ acts on the {\it right} : 
the symbol $\varphi\psi$ means that we apply $\varphi$ first 
and then $\psi$ for $\varphi, \psi \in \mathcal{M}_g$. 

Let $f\co M\rightarrow B$ be a Lefschetz fibration of genus $g$ 
as in Definition \ref{LF}. Take a base point $b_0\in B-\Delta$ 
and an orientation preserving diffeomorphism $\Phi\co\Sigma_g\rightarrow F_0:=f^{-1}(b_0)$. 
Since $f$ restricted over $B-\Delta$ 
is a smooth fiber bundle with fiber $\Sigma_g$, 
we can define a homomorphism 
\[
\rho\co\pi_1(B-\Delta,b_0)\rightarrow \mathcal{M}_g 
\]
called the {\it monodromy representation} of $f$ with respect to $\Phi$ 
(see Matsumoto \cite[\S 2]{Matsumoto1996}). 
Let $\gamma$ be the loop based at $b_0$ 
consisting of the boundary circle of a small disk neighborhood 
of $b\in\Delta$ oriented counterclockwise and 
a simple path connecting a point on the circle to $b_0$ in $B-\Delta$. 
It is known that $\rho([\gamma])$ is a 
Dehn twist along some essential simple closed curve $c$ on $\Sigma_g$. 
The curve $c$ is called the {\it vanishing cycle} of 
the critical point $p$ on $f^{-1}(b)$. 
If $p$ is of positive type (resp. of negative type), then 
the Dehn twist is right-handed (resp. left-handed). 

A singular fiber is said to be of  {\it type {\rm I}} if the vanishing cycle is non-separating 
and of {\it type ${\rm II}_h$} for $h=1,\ldots , [g/2]$ if the vanishing cycle is 
separating and it bounds a genus--$h$ subsurface of $\Sigma_g$.  
A singular fiber is said to be of {\it type ${\rm I}^+$} (resp. 
{\it type ${\rm I}^-$}, {\it type ${\rm II}_h^+$},  {\it type ${\rm II}_h^-$}) 
if it is of type I and of positive type (resp. of type I and of negative type, 
of type ${\rm II}_h$ and of positive type, of type ${\rm II}_h$ and of negative type). 
We denote by 
$n_0^{+}(f)$, $n_0^{-}(f)$, $n_h^{+}(f)$, and $n_h^{-}(f)$, 
the numbers of singular fibers of $f$ of type 
${\rm I}^+$, ${\rm I}^-$, ${\rm II}_h^+$, and ${\rm II}_h^-$, respectively.  
A Lefschetz fibration is called  {\it irreducible} if 
every singular fiber is of type I.   
A Lefschetz fibration is called {\it chiral} if every singular fiber is of positive type. 

Suppose that the cardinality of $\Delta$ is equal to $n$. 
A system $\mathcal{A}= (A_1, \dots, A_n)$ of arcs on $B$ is called 
a {\it Hurwitz arc system} for $\Delta$ with base point $b_0$ if 
each $A_i$ is an embedded arc connecting $b_0$ 
with a point of $\Delta$ in $B$ such that $A_i \cap A_j= \{b_0\}$ for $i \ne j$, 
and they appear in this order around $b_0$ (see Kamada \cite{Kamada2002}).  
When $B$ is a $2$-sphere, 
the system $\mathcal{A}$ determines a system of generators of 
$\pi_1(B- \Delta, b_0)$, say $(a_1, \dots, a_n)$.  
We call $( \rho(a_1), \dots, \rho(a_n) )$ a {\it Hurwitz system} of $f$.  
It is easy to see that $\rho$ is determined by $( \rho(a_1), \dots, \rho(a_n) )$. 

\subsection{Hyperelliptic structures}

We next introduce a notion of hyperellipticity for Lefschetz fibrations (cf. 
Endo \cite[Definition 4.2]{Endo2000}). 

Let $\iota$ be the isotopy class of a {\it hyperelliptic involution} $I$, 
an involution on $\Sigma_g$ with $2g+2$ fixed points, 
and $\mathcal{H}_g$ the centralizer of $\iota$ in $\mathcal{M}_g$, 
which is called the {\it hyperelliptic mapping class group} of $\Sigma_g$ with respect to $\iota$. 
Let $\zeta_1,\ldots ,\zeta_{2g+1},\sigma_1,\ldots ,\sigma_{[g/2]}$ be 
right-handed Dehn twists along simple closed curves 
$c_1,\ldots ,c_{2g+1}, s_1,\ldots ,s_{[g/2]}$ on $\Sigma_g$ 
depicted in Figure \ref{curves1}, respectively. 
We suppose that $c_1,\ldots ,c_{2g+1}, s_1,\ldots ,s_{[g/2]}$ are invariant 
under $I$ and thus 
$\zeta_1,\ldots ,\zeta_{2g+1},\sigma_1,\ldots ,\sigma_{[g/2]}$ belong to 
$\mathcal{H}_g$. 

\begin{figure}[ht!]
\labellist
\small \hair 2pt
\pinlabel $c_1$ [r] at 0 155
\pinlabel $c_2$ [t] at 45 139
\pinlabel $c_3$ [t] at 70 146
\pinlabel $c_4$ [t] at 98 139
\pinlabel $c_{2g}$ [t] at 242 139
\pinlabel $c_{2g+1}$ [l] at 285 155
\pinlabel $c_2$ [t] at 44 26
\pinlabel $c_{2h}$ [t] at 105 26
\pinlabel $c_{2h+2}$ [t] at 182 26
\pinlabel $c_{2g}$ [t] at 243 26
\pinlabel $s_{h}$ [r] at 135 72
\endlabellist
\centering
\includegraphics[scale=0.65]{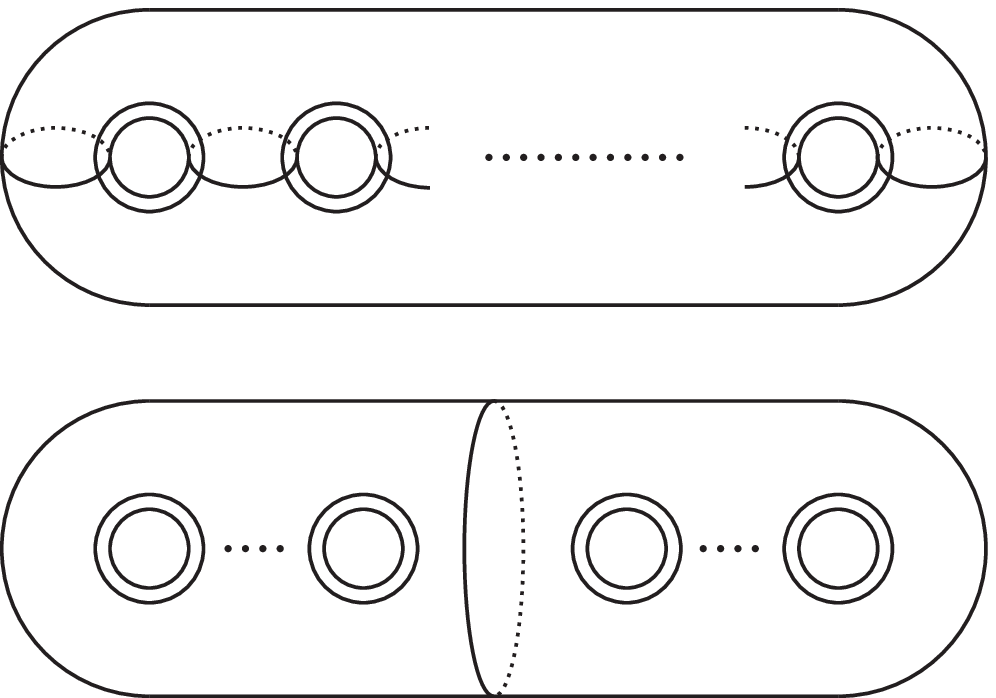}
\caption{Simple closed curves on $\Sigma_g$}
\label{curves1}
\end{figure}

\begin{defn}\label{HLF} 
Let $f\co M\rightarrow B$ be a Lefschetz fibration of genus $g$, 
$\Delta$ the set of critical values of $f$, 
and $b_0\in B-\Delta$ a base point. 
Take an orientation preserving diffeomorphism $\Phi\co\Sigma_g\rightarrow f^{-1}(b_0)$ 
and consider the monodromy representation 
$\rho\co\pi_1(B-\Delta,b_0)\rightarrow \mathcal{M}_g$ of $f$ with respect to $\Phi$. 
The pair $(f,\Phi)$ is called a {\it hyperelliptic Lefschetz fibration} 
(and $\Phi$ is called a {\it hyperelliptic structure} on $f$) 
if the image of $\rho$ is included in $\mathcal{H}_g$. 
Let $\Phi, \Phi'\co\Sigma_g\rightarrow f^{-1}(b_0)$ be hyperelliptic structures on $f$. 
We say that $\Phi$ is {\it equivalent} to $\Phi'$ 
if the isotopy class of $(\Phi')^{-1}\circ \Phi$ belongs to $\mathcal{H}_g$. 
\end{defn}

\begin{defn}\label{H-isom} 
Let $f\co M\rightarrow B$ and $f'\co M'\rightarrow B$ be Lefschetz fibrations of genus $g$ 
over the same base space $B$, 
$\Delta$ and $\Delta'$ the sets of critical values of $f$ and $f'$, 
and $b_0\in B-\Delta$ and $b'_0\in B-\Delta'$ base points for $f$ and $f'$, respectively. 
Take orientation preserving diffeomorphisms $\Phi\co\Sigma_g\rightarrow f^{-1}(b_0)$ and 
$\Phi'\co\Sigma_g\rightarrow f'^{-1}(b'_0)$ and 
suppose that the pairs $(f,\Phi)$ and $(f',\Phi')$ are hyperelliptic Lefschetz fibrations. 
We say that $(f,\Phi)$ is $\mathcal{H}$--{\it isomorphic} to $(f',\Phi')$ 
if there exist orientation preserving diffeomorphisms $H\co M\rightarrow M'$ 
and $h\co B\rightarrow B$ which satisfy the following conditions: 
(i) $f'\circ H=h\circ f$; (ii) $h(b_0)=b'_0$; 
(iii) $H|_{f^{-1}(b_0)}\circ \Phi$ is a hyperelliptic structure on $f'$ equivalent to $\Phi'$. 
If we can choose such an $h$ isotopic to the identity relative to a given base point 
$b_0$, we say that $(f,\Phi)$ is {\it strictly} $\mathcal{H}$--{\it isomorphic} to $(f',\Phi')$. 
If $(f,\Phi)$ is $\mathcal{H}$--isomorphic (resp. strictly $\mathcal{H}$--isomorphic) 
to $(f',\Phi')$, then $f$ is isomorphic (resp. strictly isomorphic) to $f'$. 
\end{defn}

The next lemmas easily follow from theorems of Matsumoto 
\cite[Theorem 2.4, Theorem 2.6]{Matsumoto1996} 
(see also Kas \cite[Theorem 2.4]{Kas1980}). 

\begin{lem}\label{monod} 
Suppose that $g$ is greater than one. 
Let $(f,\Phi)$ and $(f',\Phi')$ be hyperelliptic Lefschetz fibrations of genus $g$ 
as in Definition \ref{H-isom}, 
and $\rho\co\pi_1(B-\Delta,b_0)\rightarrow \mathcal{H}_g$ and 
$\rho'\co\pi_1(B-\Delta',b'_0)\rightarrow \mathcal{H}_g$ the 
monodromy representations of $f$ and $f'$ with respect to $\Phi$ and $\Phi'$, 
respectively. 

{\rm (1)} $(f,\Phi)$ is $\mathcal{H}$--isomorphic to $(f',\Phi')$ 
if and only if there exist an orientation preserving diffeomorphism 
$h\co B\rightarrow B$ and an element $\alpha$ of $\mathcal{H}_g$ 
which satisfies the following conditions: 
{\rm (i)} $h(\Delta)=\Delta';$ {\rm (ii)} $h(b_0)=b'_0;$ 
{\rm (iii)} $\rho'(h_{\#}(\gamma))=\alpha^{-1}\rho(\gamma)\alpha$ for every 
$\gamma\in\pi_1(B-\Delta,b_0)$. 

{\rm (2)} $(f,\Phi)$ is strictly $\mathcal{H}$--isomorphic to $(f',\Phi')$ 
if and only if there exist an orientation preserving diffeomorphism 
$h\co B\rightarrow B$ and an element $\alpha$ of $\mathcal{H}_g$ 
which satisfies the conditions {\rm (i), (ii), (iii)}, and 
{\rm (iv)} $h$ is isotopic to the identity relative to a given base point $b_0$. 
\end{lem}

\begin{lem}\label{exist1}
Suppose that $g$ is greater than one. 
Let $\rho\co\pi_1(B-\Delta,b_0)\rightarrow \mathcal{H}_g$ be a homomorphism and 
$\mathcal{A}= (A_1, \dots, A_n)$ a Hurwitz arc system for $\Delta$ with base point $b_0$. 
We assume that $\rho(a_1), \dots, \rho(a_n)$ are Dehn twists 
along simple closed curves on $\Sigma_g$ for the system of generators 
$(a_1, \dots, a_n)$ of $\pi_1(B- \Delta, b_0)$ determined by $\mathcal{A}$. 
Then there exists a hyperelliptic Lefschetz fibration $(f,\Phi)$ of genus $g$ 
as in Definition \ref{HLF} 
with monodromy representation $\rho$. 
\end{lem}

For any $(\alpha_1,\ldots ,\alpha_n)\in (\mathcal{H}_g)^n$ such that each $\alpha_i$ is 
a Dehn twist, there exists a hyperelliptic Lefschetz fibration $(f,\Phi)$ of genus $g$ 
over $S^2$ with Hurwitz system $(\alpha_1,\ldots ,\alpha_n)$ by Lemma \ref{exist1}. 
We call such $(f,\Phi)$ a hyperelliptic Lefschetz fibration {\it determined} by 
$(\alpha_1,\ldots ,\alpha_n)$. 

Two isomorphic hyperelliptic Lefschetz fibrations need not be $\mathcal{H}$--isomorphic. 
We give a sufficient condition for isomorphic hyperelliptic Lefschetz fibrations to be 
$\mathcal{H}$--isomorphic. 

Let $(f,\Phi)$ and $(f',\Phi')$ be hyperelliptic Lefschetz fibrations of genus $g$ 
as in Definition \ref{H-isom}, and $\rho$ and $\rho'$ their 
monodromy representations as in Lemma \ref{monod}. 

\begin{prop}\label{image} 
Suppose that the image of $\rho$ is $\mathcal{H}_g$. 
Then $f$ is isomorphic to $f'$ if and only if 
$(f,\Phi)$ is $\mathcal{H}$--isomorphic to $(f',\Phi')$. 
\end{prop}

\noindent
{\it Proof}. 
The `if' part is obvious. We show the `only if' part. 
Since $f$ is isomorphic to $f'$, 
there exist an orientation preserving diffeomorphism 
$h\co B\rightarrow B$ and an element $\alpha$ of $\mathcal{M}_g$ 
which satisfies the following conditions (see Matsumoto 
\cite[Theorem 2.4]{Matsumoto1996}): 
{\rm (i)} $h(\Delta)=\Delta';$ {\rm (ii)} $h(b_0)=b'_0;$ 
{\rm (iii)} $\rho'(h_{\#}(\gamma))=\alpha^{-1}\rho(\gamma)\alpha$ for every 
$\gamma\in\pi_1(B-\Delta,b_0)$. 
We will show that $\alpha$ belongs to $\mathcal{H}_g$, 
which implies that $(f,\Phi)$ is $\mathcal{H}$--isomorphic to $(f',\Phi')$ by Lemma \ref{monod}. 

For any $i=1,\ldots ,2g+1$, there exists an element $\gamma_i$ of $\pi_1(B-\Delta, b_0)$ 
such that $\rho(\gamma_i)=\zeta_i$ because the image of $\rho$ is 
$\mathcal{H}_g$. Then we have 
\[
\alpha^{-1}\zeta_i\alpha=\alpha^{-1}\rho(\gamma_i)\alpha
=\rho'(h_{\#}(\gamma_i))\in\mathcal{H}_g
\]
from (iii). This implies that $\alpha^{-1}\zeta_i\alpha=\iota^{-1}\alpha^{-1}\zeta_i\alpha\iota$ 
and thus $(c_i)A$ is isotopic to $((c_i)A)I$ by \cite[Fact 3.6]{FM2011}, 
where $A$ is an orientation preserving diffeomorphism on $\Sigma_g$
representing $\alpha$. Since $c_i$ is invariant under $I$, 
$(c_i)(A I A^{-1} I^{-1})$ is isotopic to $c_i$. 
Hence there exists a diffeomorphism $F$ on $\Sigma_g$ isotopic to the identity 
such that $(c_i)(A I A^{-1} I^{-1} F)=c_i$ for every $i$ by virtue of 
\cite[Lemma 2.9]{FM2011}. 
Since $\Sigma_g-(c_1\cup\cdots\cup c_{2g+1})$ is a disjoint union of two open disks, 
$A I A^{-1} I^{-1} F$ is isotopic to either the identity or $I$ 
by \cite[Proposition 2.8]{FM2011}. 
If $A I A^{-1} I^{-1} F$ is isotopic to $I$, then we obtain $\iota=1$, 
which is a contradiction. Therefore $A I A^{-1} I^{-1} F$ is isotopic to 
the identity and we have $\alpha\in\mathcal{H}_g$. 
This completes the proof. 
$\square$

%%%%%%%%%%%%% Section 3 %%%%%%%%%%%%%%%

\section{Chart descriptions}

%%%%%%%%%%%%%%%%%%%%%%%%%%%%%%%%%

In this section we introduce chart descriptions for hyperelliptic Lefschetz fibrations 
by employing finite presentations of hyperelliptic mapping class groups 
and two other groups. 
General theories of charts for presentations of groups were developed independently by 
Kamada \cite{Kamada2007} and Hasegawa \cite{Hasegawa2006}. 
We use the terminology of chart description in Kamada \cite{Kamada2007}. 

\subsection{Three finite presentations}\label{three}

We first review finite presentations of three groups 
related with hyperelliptic Lefschetz fibrations. 

Fadell and Van Buskirk \cite{FV1962} proved that the braid group of a $2$--sphere 
is just the braid group of a $2$--disk with a single additional relation. 

\begin{thm}[(Fadell--Van Buskirk \cite{FV1962})]
Suppose that $g$ is positive. 
The $(2g+2)$--string braid group $B_{2g+2}(S^2)$ of a $2$--sphere 
is generated by elements 
$x_1,x_2,\ldots ,x_{2g+1}$ and has defining relations: 

$(1) \quad x_ix_j =x_jx_i \;\; (1\leq i<j-1\leq 2g);$ 

$(2) \quad x_ix_{i+1}x_i =x_{i+1}x_ix_{i+1} \;\; (i=1,\ldots ,2g);$ 

$(3) \quad x_1x_2\cdots x_{2g+1}x_{2g+1}\cdots x_2x_1=1.$ 
\end{thm}

Magnus \cite{Magnus1934} obtained a presentation of the mapping class group of 
a $2$--sphere with marked points, which is essentially the same as the following one. 

\begin{thm}[(Magnus \cite{Magnus1934}, cf. Birman--Hilden \cite{BH1971})]
Suppose that $g$ is positive. 
The mapping class group $\mathcal{M}_{0,2g+2}$ of a $2$--sphere 
with $2g+2$ marked points is generated by elements 
$\xi_1,\xi_2,\ldots ,\xi_{2g+1}$ and has defining relations: 

$(1) \quad \xi_i\xi_j =\xi_j\xi_i \;\; (1\leq i<j-1\leq 2g);$ 

$(2) \quad \xi_i\xi_{i+1}\xi_i =\xi_{i+1}\xi_i\xi_{i+1} \;\; (i=1,\ldots ,2g);$ 

$(3) \quad \xi_1\xi_2\cdots\xi_{2g+1}\xi_{2g+1}\cdots\xi_2\xi_1=1;$ 

$(4) \quad (\xi_1\xi_2\cdots\xi_{2g+1})^{2g+2}=1.$ 
\end{thm}

Birman and Hilden \cite{BH1971} considered the mapping class group of the orbit space 
of the involution $I$ on $\Sigma_g$, which is a $2$--sphere with $2g+2$ marked points, 
to obtain a presentation of the hyperelliptic 
mapping class group of $\Sigma_g$. 

\begin{thm}[(Birman--Hilden \cite{BH1971})]\label{pres} 
Suppose that $g$ is positive. 
The hyperelliptic mapping class group $\mathcal{H}_g$ of $\Sigma_g$ 
is generated by elements 
$\zeta_1,\zeta_2,\ldots ,\zeta_{2g+1}$ and has defining relations: 

$(1) \quad \zeta_i\zeta_j =\zeta_j\zeta_i \;\; (1\leq i<j-1\leq 2g);$ 

$(2) \quad \zeta_i\zeta_{i+1}\zeta_i =\zeta_{i+1}\zeta_i\zeta_{i+1} \;\; (i=1,\ldots ,2g);$ 

$(3) \quad (\zeta_1\zeta_2\cdots\zeta_{2g+1}\zeta_{2g+1}\cdots\zeta_2\zeta_1)^2=1;$ 

$(4) \quad (\zeta_1\zeta_2\cdots\zeta_{2g+1})^{2g+2}=1;$ 

$(5) \quad [\zeta_1,\zeta_1\zeta_2\cdots\zeta_{2g+1}\zeta_{2g+1}\cdots\zeta_2\zeta_1]=1.$
\end{thm}

Let $g$ be a positive integer. 
Since the homomorphism from $B_{2g+2}(S^2)$ to $\mathcal{M}_{0,2g+2}$ which sends 
$x_i$ to $\xi_i$ is surjective, we have a central extension 
\[
1\longrightarrow \mathbb{Z}_2\longrightarrow B_{2g+2}(S^2)
\longrightarrow \mathcal{M}_{0,2g+2}\longrightarrow 1
\]
of $\mathcal{M}_{0,2g+2}$ with kernel generated by the {\it Dirac braid} 
$\Delta_{2g+2}=(x_1x_2\cdots x_{2g+1})^{2g+2}$ 
(see Gillette and Van Buskirk \cite{GV1968}). 
We have another central extension 
\[
1\longrightarrow \mathbb{Z}_2\longrightarrow \mathcal{H}_g
\overset{\pi}{\longrightarrow} \mathcal{M}_{0,2g+2}\longrightarrow 1
\]
of $\mathcal{M}_{0,2g+2}$ with kernel generated by the isotopy class 
$\iota=\zeta_1\zeta_2\cdots\zeta_{2g+1}\zeta_{2g+1}\cdots\zeta_2\zeta_1$ of 
the hyperelliptic involution $I$, 
where the homomorphism $\pi:\mathcal{H}_g\rightarrow \mathcal{M}_{0,2g+2}$ sends 
$\zeta_i$ to $\xi_i$.

\subsection{Chart descriptions}\label{chartdes} 

We make use of the above presentations to introduce notions of chart 
which give graphic description of monodromy representations of Lefschetz fibrations. 
We first recall a general definition of chart given by Kamada \cite{Kamada2007}. 

Let $\mathcal{X}$ be a set and $\mathcal{R}$ and $\mathcal{S}$ sets of words in 
$\mathcal{X}\cup\mathcal{X}^{-1}$. 
Let $\mathcal{C}:=(\mathcal{X},\mathcal{R},\mathcal{S})$ be 
the triple consisting of $\mathcal{X}$, $\mathcal{R}$ and $\mathcal{S}$, 
and $G$ the group with presentation $\langle \mathcal{X}\, |\, \mathcal{R}\rangle$. 
Let $B$ be a connected closed oriented surface 
and $\Gamma$ a finite graph in $B$ such that each edge of $\Gamma$ is oriented 
and labeled an element of $\mathcal{X}$. 
Choose a simple path $\gamma$ which intersects with edges of $\Gamma$ 
transversely and does not intersect with vertices of $\Gamma$. 
For such a path $\gamma$, we obtain a word 
$w_{\Gamma}(\gamma)$ in $\mathcal{X}\cup\mathcal{X}^{-1}$ 
by reading off the labels of intersecting edges along $\gamma$ with exponents 
as in Figure \ref{intersection} (a). 
We call the word $w_{\Gamma}(\gamma)$ the {\it intersection word} of $\gamma$ 
with respect to $\Gamma$. 
Conversely, we can specify the number, orientations, and labels of consecutive edges in $\Gamma$ 
by indicating a (dashed) arrow intersecting the edges transversely 
together with the intersection 
word of the arrow with respect to $\Gamma$ (see Figure \ref{intersection} (b) and (c)). 

\begin{figure}[ht!]
\labellist
\small \hair 2pt
\pinlabel $\gamma$ [r] at 0 18
\pinlabel $x$ [t] at 30 0
\pinlabel $y$ [t] at 59 0
\pinlabel $x$ [t] at 87 0
\pinlabel $z$ [t] at 116 0
\pinlabel $y$ [t] at 143 0
\pinlabel {(a)} [t] at 85 -15
\pinlabel $w$ [l] at 255 18
\pinlabel {(b)} [t] at 235 -15
\pinlabel $w$ [b] at 313 33
\pinlabel {(c)} [t] at 313 -15
\endlabellist
\centering
\includegraphics[scale=0.8]{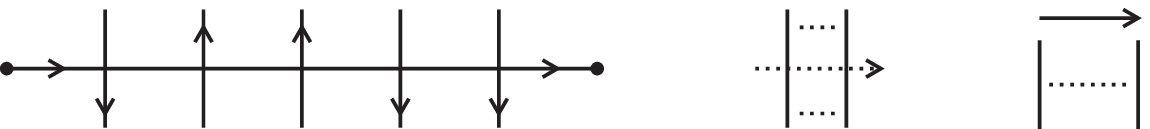}
\vspace{8mm}
\caption{Intersection word 
$w_{\Gamma}(\gamma)=w=xy^{-1}x^{-1}zy$}
\label{intersection}
\end{figure}

For a vertex $v$ of $\Gamma$, a small simple closed curve surrounding $v$ in the 
counterclockwise direction is called a {\it meridian loop} of $v$ and denoted by $m_v$. 
The vertex $v$ is said to be {\it marked} if one of the regions around 
$v$ is specified by an asterisk. 
If $v$ is marked, the intersection word $w_{\Gamma}(m_v)$ of $m_v$ 
with respect to $\Gamma$ is well-defined. 
If not, it is determined up to cyclic permutation. See Kamada \cite{Kamada2007} 
for details. 

\begin{defn}\label{def:chart} 
A $\mathcal{C}$--{\it chart} in $B$ is a finite graph $\Gamma$ in $B$ 
(possibly being empty or having {\it hoops} that are closed edges without vertices) 
whose edges are labeled an element of $\mathcal{X}$, 
and oriented so that the following conditions are satisfied:   
\begin{itemize}
\item[(1)] the vertices of $\Gamma$ are classified into two families: 
{\it white vertices} and {\it black vertices};  
\item[(2)] if $v$ is a white vertex (resp. a black vertex), 
the word $w_{\Gamma}(m_v)$ is a cyclic permutation of an element of 
$\mathcal{R}\cup\mathcal{R}^{-1}$ (resp. of $\mathcal{S}$). 
\end{itemize}
A white vertex $v$ is said to be of 
{\it type} $r$ (resp. of {\it type} $r^{-1}$) if 
$w_{\Gamma}(m_v)^{-1}$ is a cyclic permutation of $r\in\mathcal{R}$ 
(resp. of $r^{-1}\in\mathcal{R}^{-1}$). 
A black vertex $v$ is said to be of {\it type} $s$ if 
$w_{\Gamma}(m_v)$ is a cyclic permutation of $s\in\mathcal{S}$. 
A $\mathcal{C}$--chart $\Gamma$ is said to be {\it marked} 
if each white vertex (resp. black vertex) $v$ is 
marked and $w_{\Gamma}(m_v)$ is exactly an element of 
$\mathcal{R}\cup\mathcal{R}^{-1}$ (resp. of $\mathcal{S}$). 
If a base point $b_0$ of $B$ is specified, we always assume that a chart $\Gamma$ is 
disjoint from $b_0$. 
A chart consisting of two black vertices and one edge connecting them is called 
a {\it free edge}. 
A {\it subchart} of a $\mathcal{C}$--chart $\Gamma$ is the intersection of $\Gamma$ 
with a compact $2$--dimensional submanifold of $B$. 
\end{defn}

\begin{rem} 
It would be worth noting that the intersection word of a `clockwise' meridian of 
a white vertex of type $r$ is equal to $r$, while that of a `counterclockwise' meridian 
of a black vertex of type $s$ is equal to $s$ in this paper (see also \cite{EHKT2014}). 
\end{rem}

Let $\Gamma$ be a $\mathcal{C}$--chart in $B$ with base point $b_0$ 
and $\Delta_{\Gamma}$ the set of black vertices of $\Gamma$. 

\begin{defn} 
For a loop $\eta$ in $B-\Delta_{\Gamma}$ based at $b_0$, 
the element of $G$ determined by 
the intersection word $w_{\Gamma}(\eta)$ of $\eta$ with respect to $\Gamma$ does not 
depend on a choice of representative of the homotopy class of $\eta$. 
Thus we obtain a homomorphism 
$\rho_{\Gamma}\co\pi_1(B-\Delta_{\Gamma},b_0)\rightarrow G$, 
which is called the {\it homomorphism determined by $\Gamma$}. 
\end{defn}

Let $\Delta$ be a finite subset of $B$ and $b_0$ a base point of $B-\Delta$. 

\begin{defn}
A homomorphism from $\pi_1(B-\Delta,b_0)$ to $G$ is called 
a $G$--{\it monodromy representation}. 
A loop $\ell$ in $B-\Delta$ based at $b_0$ is called a {\it meridional loop} 
if it is obtained from a meridian loop $m_v$ of a point $v$ of $\Delta$ 
by connecting with $b_0$ along a path $n$ in $B-\Delta$, 
namely: $\ell=n^{-1}\cdot m_v\cdot n$. 
We denote by $\mathcal{M}(B,\Delta, b_0;\mathcal{C})$ the set of $G$--monodromy 
representations $\rho:\pi_1(B-\Delta, b_0)\rightarrow G$ such that 
$\rho([\ell])$ is conjugate to the element of $G$ determined by an element of 
$\mathcal{S}$ for every meridional loop $\ell$. 
For a $\mathcal{C}$--chart $\Gamma$ in $B$ with base point $b_0$ and the black vertices 
$\Delta_{\Gamma}$, the homomorphism $\rho_{\Gamma}$ determined by $\Gamma$ 
belongs to $\mathcal{M}(B,\Delta_{\Gamma}, b_0;\mathcal{C})$. 
\end{defn}

\begin{thm}[(Kamada \cite{Kamada2007}, Hasegawa \cite{Hasegawa2006})]\label{exist} 
For any $G$--monodromy representation $\rho:\pi_1(B-\Delta, b_0)\rightarrow G$ 
belonging to $\mathcal{M}(B,\Delta, b_0;\mathcal{C})$, 
there exists a $\mathcal{C}$--chart $\Gamma$ such that $\rho_{\Gamma}=\rho$. 
\end{thm}

We next introduce several moves for charts. 
Let $\Gamma$ and $\Gamma'$ be two 
$\mathcal{C}$--charts on $B$ and $b_0$ a base point of $B$. 

Let $D$ be a disk embedded in $B-\{b_0\}$. 
Suppose that the boundary $\partial D$ of $D$ intersects $\Gamma$ and $\Gamma'$ 
transversely. 

\begin{defn} 
We say that $\Gamma'$ is obtained from $\Gamma$ by a {\it chart move of type W} 
if $\Gamma\cap (B-{\rm Int}\, D)=\Gamma'\cap (B-{\rm Int}\, D)$ and 
that both $\Gamma\cap D$ and $\Gamma'\cap D$ have no black vertices. 
We call chart moves of type W shown in Figure \ref{movesA} (a), (b), and (c), 
a {\it channel change}, a {\it birth/death of a hoop}, and 
a {\it birth/death of a pair of white vertices}, respectively. 
\end{defn}

\begin{figure}[ht!]
\labellist
\footnotesize \hair 2pt
\pinlabel $x$ [tr] at 0 66
\pinlabel $x$ [bl] at 41 106
\pinlabel $x$ [tr] at 80 66
\pinlabel $x$ [bl] at 121 106
\pinlabel (a) [t] at 61 84
\pinlabel $x$ [tl] at 184 76
\pinlabel empty at 258 86
\pinlabel (b) [t] at 209 84
\pinlabel $r$ at 57 20
\pinlabel $r^{\negthinspace -\negthinspace 1}$ at 87 23
\pinlabel (c) [t] at 132 19
\endlabellist
\centering
\includegraphics[scale=0.8]{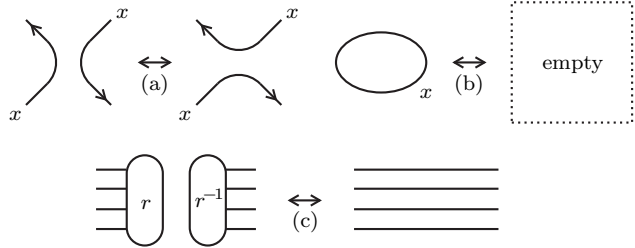}
\caption{Chart moves of type W}
\label{movesA}
\end{figure}

Let $s$ and $s'$ be elements of $\mathcal{S}$. 
Suppose that there exists a word $w$ in $\mathcal{X}\cup\mathcal{X}^{-1}$ such that 
two words $s'$ and $wsw^{-1}$ determine the same element of $G$. 

\begin{defn} 
If a $\mathcal{C}$--chart $\Gamma$ contains a black vertex of type $s$, 
then we can change a part of $\Gamma$ near the vertex by using a local replacement 
depicted in Figure \ref{movesB} to obtain another $\mathcal{C}$--chart $\Gamma'$. 
We say that $\Gamma'$ is obtained from $\Gamma$ by a {\it chart move of transition}. 
Note that the box labeled $T$ can be filled only with edges and white vertices. 
\end{defn}

\begin{figure}[ht!]
\labellist
\footnotesize \hair 2pt
\pinlabel $s$ [b] at 35 81
\pinlabel $s'$ [b] at 205 81
\pinlabel $w$ [bl] at 245 86
\pinlabel $w$ [tl] at 245 40
\pinlabel $s$ [bl] at 257 81
\pinlabel $T$ at 232 64
\endlabellist
\centering
\includegraphics[scale=0.8]{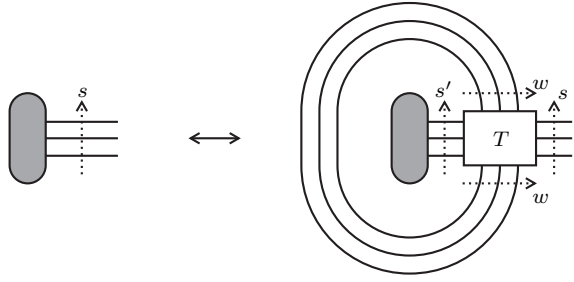}
\caption{Chart move of transition}
\label{movesB}
\end{figure}

\begin{defn} 
We say that $\Gamma'$ is obtained from $\Gamma$ by a {\it chart move of conjugacy type} 
if $\Gamma'$ is obtained from $\Gamma$ 
by a local replacement depicted in Figure \ref{movesC}. 
\end{defn}

\begin{figure}[ht!]
\labellist
\footnotesize \hair 2pt
\pinlabel $b_0$ [l] at 7 22
\pinlabel $b_0$ [l] at 117 22
\pinlabel $x$ [l] at 135 22
\pinlabel $b_0$ [l] at 197 22
\pinlabel $b_0$ [l] at 304 22
\pinlabel $x$ [l] at 323 22
\endlabellist
\centering
\includegraphics[scale=0.7]{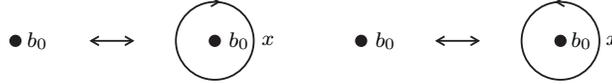}
\caption{Chart moves of conjugacy type}
\label{movesC}
\end{figure}

Let $\Delta$ and $\Delta'$ be finite subsets of $B$, and $b_0$ and $b'_0$ base points of 
$B-\Delta$ and $B-\Delta'$, respectively. 

\begin{defn} 
Let $\rho:\pi_1(B-\Delta, b_0)\rightarrow G$ and $\rho':\pi_1(B-\Delta',b'_0)\rightarrow G$ 
be $G$--monodromy representations. 
We say that $\rho$ is {\it equivalent} to $\rho'$ 
if there exist an orientation preserving diffeomorphism 
$h\co B\rightarrow B$ and an element $\alpha$ of $G$ 
which satisfies the following conditions: 
{\rm (i)} $h(\Delta)=\Delta';$ {\rm (ii)} $h(b_0)=b'_0;$ 
{\rm (iii)} $\rho'(h_{\#}(\gamma))=\alpha^{-1}\rho(\gamma)\alpha$ for every 
$\gamma\in\pi_1(B-\Delta,b_0)$. 
If we can choose such an $h$ isotopic to the identity relative to a given base point 
$b_0$, we say that $\rho$ is {\it strictly equivalent} to $\rho'$. 
\end{defn}

We state a classification theorem for $G$--monodromy representations in terms of charts and 
chart moves. 

\begin{thm}[(Kamada \cite{Kamada2007}, Hasegawa \cite{Hasegawa2006})]\label{classif} 
Let $\Gamma$ and $\Gamma'$ be $\mathcal{C}$--charts in $B$, 
and $\rho_{\Gamma}:\pi_1(B-\Delta_{\Gamma}, b_0)\rightarrow G$ 
and $\rho_{\Gamma'}:\pi_1(B-\Delta_{\Gamma'},b'_0)\rightarrow G$ 
the homomorphisms determined by $\Gamma$ and $\Gamma'$, respectively. 

{\rm (1)} $\rho_{\Gamma}$ is equivalent to $\rho_{\Gamma'}$ 
if and only if $\Gamma$ is transformed into $\Gamma'$ by a finite sequence of 
the following moves: 
{\rm (i)} chart moves of type W; 
{\rm (ii)} chart moves of transition; 
{\rm (iii)} chart moves of conjugacy type; 
and 
{\rm (iv)} sending $\Gamma$ to $\Gamma'$ by 
orientation preserving diffeomorphisms $h:B\rightarrow B$ which 
satisfies $h(b_0)=b'_0$. 

{\rm (2)} $\rho_{\Gamma}$ is strictly equivalent to $\rho_{\Gamma'}$ 
if and only if $\Gamma$ is transformed into $\Gamma'$ by a finite sequence of 
the moves {\rm (i), (ii), (iii),} and {\rm (iv)} provided that 
$h$ is isotopic to the identity relative to a given base point $b_0$. 
\end{thm}

We now define three explicit $\mathcal{C}$s corresponding to three groups 
$\mathcal{M}_{0,2g+2}$, $B_{2g+2}(S^2)$, and $\mathcal{H}_g$. 

For the group $\mathcal{M}_{0,2g+2}$, we set 
%{\allowdisplaybreaks %
\begin{align*}
\mathcal{C}_0 & :=(\mathcal{X}_0,\mathcal{R}_0,\mathcal{S}_0), \\
\mathcal{X}_0 & :=\{\xi_1,\xi_2,\ldots ,\xi_{2g+1}\}, \\
\mathcal{R}_0 & :=\{r_1(i,j)\, |\, 1\leq i<j-1\leq 2g\}
\cup \{r_2(i)\, |\, i=1,\ldots ,2g\}\cup \{r_3,r_4 \}, \\
\mathcal{S}_0 & :=\{\ell_0(i)^{\pm 1}\, |\, i=1,\ldots ,2g+1\} 
\cup \{\ell_h^{\pm 1}\, |\, h=1,\ldots ,[g/2]\}
\end{align*}
%}
for $g\geq 1$, where 
%{\allowdisplaybreaks %
\begin{align*}
& r_1(i,j) :=\xi_i\xi_j\xi_i^{-1}\xi_j^{-1} \;\; (1\leq i<j-1\leq 2g), \\
& r_2(i) :=\xi_i\xi_{i+1}\xi_i\xi_{i+1}^{-1}\xi_i^{-1}\xi_{i+1}^{-1} \;\; (i=1,\ldots ,2g), \\
& r_3:=\xi_1\xi_2\cdots\xi_{2g+1}\xi_{2g+1}\cdots\xi_2\xi_1, 
\quad r_4:=(\xi_1\xi_2\cdots\xi_{2g+1})^{2g+2}, \\
& \ell_0(i):=\xi_i \;\; (i=1,\ldots ,2g+1), 
\quad \ell_h:=(\xi_1\xi_2\cdots\xi_{2h})^{4h+2} \;\; (h=1,\ldots ,[g/2]). 
\end{align*}
%}
The relator $r_4$ corresponds to the Dirac braid 
and it is called the {\it Dirac braid relator}. 
Vertices of types $\ell_0(i)^{\pm 1}$, $r_1(i,j)$, $r_2(i)$, $r_3$, $r_4$, $\ell_h$ 
in $\mathcal{C}_0$--charts 
are depicted in Figure  \ref{verticesA} and Figure \ref{verticesB}, 
where the label $\xi_i$ is denoted by $i$ for short. 

\begin{figure}[ht!]
\labellist
\footnotesize \hair 2pt
\pinlabel $i$ [l] at 46 29
\pinlabel $i$ [l] at 46 5
\pinlabel $i$ [tr] at 105 3
\pinlabel $j$ [br] at 105 33
\pinlabel $i$ [bl] at 136 33
\pinlabel $j$ [tl] at 136 3
\pinlabel $i$ [tr] at 174 3
\pinlabel $j$ [br] at 174 33
\pinlabel $i$ [bl] at 204 33
\pinlabel $j$ [tl] at 204 3
\pinlabel $=$ [l] at 150 17
\pinlabel $i$ [tr] at 268 0
\pinlabel $i\negthinspace +\negthinspace 1$ [r] at 261 17
\pinlabel $i$ [br] at 270 34
\pinlabel $i\negthinspace +\negthinspace 1$ [bl] at 306 34
\pinlabel $i$ [l] at 313 17
\pinlabel $i\negthinspace +\negthinspace 1$ [tl] at 306 0
\endlabellist
\centering
\includegraphics[scale=0.8]{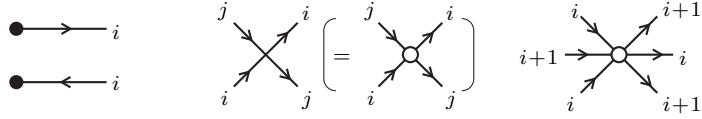}
\vspace{3mm}
\caption{Vertices of type $\ell_0(i)^{\pm 1}$, $r_1(i,j)$, $r_2(i)$}
\label{verticesA}
\end{figure}

\begin{figure}[ht!]
\labellist
\footnotesize \hair 2pt
\pinlabel $1$ [r] at 0 86
\pinlabel $2g+1$ [r] at 0 65
\pinlabel $2g+1$ [r] at 0 55
\pinlabel $1$ [r] at 0 36
\pinlabel $2g+1$ [r] at 133 112
\pinlabel $1$ [r] at 133 91
\pinlabel $2g+1$ [r] at 133 56
\pinlabel $1$ [r] at 133 37
\pinlabel $2g+1$ [r] at 133 27
\pinlabel $1$ [r] at 133 7
\pinlabel $1$ [r] at 257 114
\pinlabel $2h$ [r] at 257 93
\pinlabel $1$ [r] at 257 83
\pinlabel $2h$ [r] at 257 65
\pinlabel $1$ [r] at 257 30
\pinlabel $2h$ [r] at 257 8
\endlabellist
\centering
\includegraphics[scale=0.75]{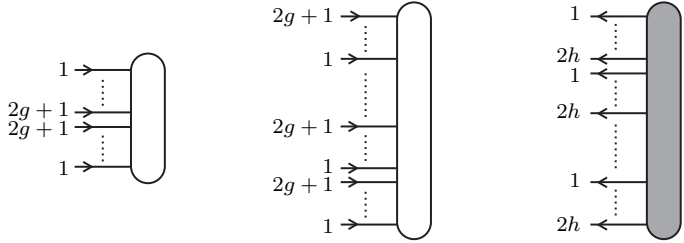}
\caption{Vertices of type $r_3$, $r_4$, $\ell_h$}
\label{verticesB}
\end{figure}
For the group $B_{2g+2}(S^2)$, we set 
%{\allowdisplaybreaks %
\begin{align*}
\tilde{\mathcal{C}} & :=(\tilde{\mathcal{X}},\tilde{\mathcal{R}},\tilde{\mathcal{S}}), \\
\tilde{\mathcal{X}} & :=\{x_1,x_2,\ldots ,x_{2g+1}\}, \\
\tilde{\mathcal{R}} & :=\{\tilde{r}_1(i,j)\, |\, 1\leq i<j-1\leq 2g\}
\cup \{\tilde{r}_2(i)\, |\, i=1,\ldots ,2g\}\cup \{\tilde{r}_3 \}, \\
\tilde{\mathcal{S}} & :=\{\tilde{\ell}_0(i)^{\pm 1}\, |\, i=1,\ldots ,2g+1\} 
\cup \{\tilde{\ell}_h^{\pm 1}\, |\, h=1,\ldots ,[g/2]\}
\end{align*}
%}
for $g\geq 1$, where 
%{\allowdisplaybreaks %
\begin{align*}
& \tilde{r}_1(i,j) :=x_ix_jx_i^{-1}x_j^{-1} \;\; (1\leq i<j-1\leq 2g), \\
& \tilde{r}_2(i) :=x_ix_{i+1}x_ix_{i+1}^{-1}x_i^{-1}x_{i+1}^{-1} \;\; (i=1,\ldots ,2g), \\
& \tilde{r}_3:=x_1x_2\cdots x_{2g+1}x_{2g+1}\cdots x_2x_1, \\
& \tilde{\ell}_0(i):=x_i \;\; (i=1,\ldots ,2g+1), 
\quad \tilde{\ell}_h:=(x_1x_2\cdots x_{2h})^{4h+2} \;\; (h=1,\ldots ,[g/2]). 
\end{align*}
%}
Vertices of types $\tilde{\ell}_0(i)^{\pm 1}$, $\tilde{r}_1(i,j)$, 
$\tilde{r}_2(i)$, $\tilde{r}_3$, $\tilde{\ell}_h$ 
in $\tilde{\mathcal{C}}$--charts 
are similar to those of types 
$\ell_0(i)^{\pm 1}$, $r_1(i,j)$, $r_2(i)$, $r_3$, $\ell_h$ 
in $\mathcal{C}_0$--charts 
(cf. Figure  \ref{verticesA} and Figure \ref{verticesB}), respectively. 

For the group $\mathcal{H}_g$, we set 
%{\allowdisplaybreaks %
\begin{align*}
\hat{\mathcal{C}} & :=(\hat{\mathcal{X}},\hat{\mathcal{R}},\hat{\mathcal{S}}), \\
\hat{\mathcal{X}} & :=\{\zeta_1,\zeta_2,\ldots ,\zeta_{2g+1}\}, \\
\hat{\mathcal{R}} & :=\{\hat{r}_1(i,j)\, |\, 1\leq i<j-1\leq 2g\}
\cup \{\hat{r}_2(i)\, |\, i=1,\ldots ,2g\}\cup \{\hat{r}_3,\hat{r}_4,\hat{r}_5 \}, \\
\hat{\mathcal{S}} & :=\{\hat{\ell}_0(i)^{\pm 1}\, |\, i=1,\ldots ,2g+1\} 
\cup \{\hat{\ell}_h^{\pm 1}\, |\, h=1,\ldots ,[g/2]\}
\end{align*}
%}
for $g\geq 1$, where 
%{\allowdisplaybreaks %
\begin{align*}
& \hat{r}_1(i,j) :=\zeta_i\zeta_j\zeta_i^{-1}\zeta_j^{-1} \;\; (1\leq i<j-1\leq 2g), \\
& \hat{r}_2(i) :=\zeta_i\zeta_{i+1}\zeta_i
\zeta_{i+1}^{-1}\zeta_i^{-1}\zeta_{i+1}^{-1} \;\; (i=1,\ldots ,2g), \\
& \hat{r}_3:=(\zeta_1\zeta_2\cdots\zeta_{2g+1}\zeta_{2g+1}\cdots\zeta_2\zeta_1)^2, 
\quad \hat{r}_4:=(\zeta_1\zeta_2\cdots\zeta_{2g+1})^{2g+2}, \\
& \hat{r}_5:=[\zeta_1,\zeta_1\zeta_2\cdots\zeta_{2g+1}\zeta_{2g+1}\cdots\zeta_2\zeta_1], \\
& \hat{\ell}_0(i):=\zeta_i \;\; (i=1,\ldots ,2g+1), 
\quad \hat{\ell}_h:=(\zeta_1\zeta_2\cdots\zeta_{2h})^{4h+2} \;\; (h=1,\ldots ,[g/2]). 
\end{align*}
%}
Vertices of types $\hat{\ell}_0(i)^{\pm 1}$, $\hat{r}_1(i,j)$, 
$\hat{r}_2(i)$, $\hat{r}_4$, $\hat{\ell}_h$ 
in $\hat{\mathcal{C}}$--charts 
are similar to those of types 
$\ell_0(i)^{\pm 1}$, $r_1(i,j)$, $r_2(i)$, $r_4$, $\ell_h$ 
in $\mathcal{C}_0$--charts 
(cf. Figure  \ref{verticesA} and Figure \ref{verticesB}), respectively. 
Vertices of types $\hat{r}_3$ and $\hat{r}_5$ 
in $\hat{\mathcal{C}}$--charts 
are depicted in Figure  \ref{verticesC}, 
where the label $\zeta_i$ is denoted by $i$ for short. 

\begin{figure}[ht!]
\labellist
\footnotesize \hair 2pt
\pinlabel $1$ [r] at 0 115
\pinlabel $2g+1$ [r] at 0 94
\pinlabel $2g+1$ [r] at 0 84
\pinlabel $1$ [r] at 0 67
\pinlabel $1$ [r] at 0 57
\pinlabel $2g+1$ [r] at 0 38
\pinlabel $2g+1$ [r] at 0 28
\pinlabel $1$ [r] at 0 10
\pinlabel $1$ [r] at 122 86
\pinlabel $2g+1$ [r] at 122 65
\pinlabel $2g+1$ [r] at 122 55
\pinlabel $1$ [r] at 122 37
\pinlabel $1$ [l] at 196 86
\pinlabel $2g+1$ [l] at 196 65
\pinlabel $2g+1$ [l] at 196 55
\pinlabel $1$ [l] at 196 37
\pinlabel $i$ [l] at 162 113
\pinlabel $i$ [l] at 162 7
\endlabellist
\centering
\includegraphics[scale=0.75]{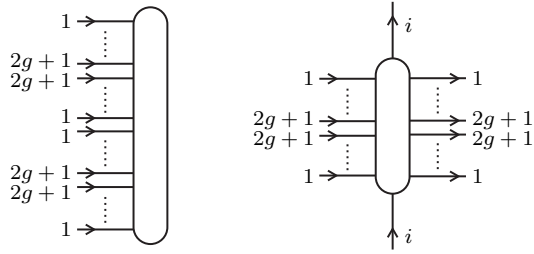}
\caption{Vertices of type $\hat{r}_3$ and $\hat{r}_5$}
\label{verticesC}
\end{figure}

Let $B$ be a connected closed oriented surface and $\Gamma$ a chart in $B$. 
We denote the number of white vertices of type $\hat{r}_1(i,j)$ 
(resp. $\hat{r}_2(i),\, \hat{r}_3,\, \hat{r}_4,\, \hat{r}_5$) 
minus the number of white vertices 
of type $\hat{r}_1(i,j)^{-1}$ 
(resp. $\hat{r}_2(i)^{-1},\, \hat{r}_3^{-1},\, \hat{r}_4^{-1},\, \hat{r}_5^{-1}$)
included in $\Gamma$ by $m_1(i,j)(\Gamma)$ 
(resp. $m_2(i)(\Gamma),\, m_3(\Gamma),\, m_4(\Gamma),\, m_5(\Gamma)$). 
Similarly, 
we denote the number of black vertices of type $\hat{\ell}_0(i)^{\pm 1}$ 
(resp. $\hat{\ell}_h^{\pm 1}$) included in $\Gamma$ 
by $n_0^{\pm}(i)(\Gamma)$ (resp. $n_h^{\pm}(\Gamma)$), 
and set $n_0(i)(\Gamma):=n_0^+(i)(\Gamma)-n_0^-(i)(\Gamma)$ 
(resp. $n_h(\Gamma):=n_h^+(\Gamma)-n_h^-(\Gamma)$) 
and $n_0^{\pm}(\Gamma):=\sum_{i=0}^{2g}n_0^{\pm}(i)(\Gamma)$.

\subsection{Charts and hyperelliptic Lefschetz fibrations} 

Combining Lemma \ref{monod} and Lemma \ref{exist1} with Theorem \ref{exist} and 
Theorem \ref{classif} for $\hat{\mathcal{C}}$--charts, 
we obtain classification theorems for hyperelliptic Lefschetz fibrations 
in terms of $\hat{\mathcal{C}}$--charts. 

\begin{prop}\label{correspondence} 
Suppose that $g$ is greater than one. 

{\rm (1)} Let $(f,\Phi)$ be a hyperelliptic Lefschetz fibration of genus $g$ over $B$ and 
$\rho$ the monodromy representation of $f$ with respect to $\Phi$. 
Then there exists a $\hat{\mathcal{C}}$--chart $\Gamma$ in $B$ 
such that the homomorphism $\rho_{\Gamma}$ 
determined by $\Gamma$ is equal to $\rho$. 

{\rm (2)} For every $\hat{\mathcal{C}}$--chart $\Gamma$ in $B$, 
there exists a hyperelliptic Lefschetz fibration 
$(f,\Phi)$ of genus $g$ over $B$ such that the monodromy representation of $f$ 
with respect to $\Phi$ 
is equal to the homomorphism $\rho_{\Gamma}$ determined by $\Gamma$. 
\end{prop}

We call such $\Gamma$ as in Proposition \ref{correspondence} (1) 
a $\hat{\mathcal{C}}$--chart {\it corresponding to $(f,\Phi)$}, 
and such $(f,\Phi)$ as in Proposition \ref{correspondence} (2) 
a hyperelliptic Lefschetz fibration {\it described by $\Gamma$}. 

\begin{prop}\label{isomorphism} 
Suppose that $g$ is greater than one. 
Let $(f,\Phi)$ and $(f',\Phi')$ be hyperelliptic Lefschetz fibrations of genus $g$, 
and $\Gamma$ and $\Gamma'$ $\hat{\mathcal{C}}$--charts in $B$ 
corresponding to $(f,\Phi)$ and $(f',\Phi')$, respectively. 

{\rm (1)} $(f,\Phi)$ is $\mathcal{H}$--isomorphic to $(f',\Phi')$ 
if and only if $\,\Gamma$ is transformed into $\Gamma'$ by a finite sequence of 
the following moves: 
{\rm (i)} chart moves of type W; 
{\rm (ii)} chart moves of transition; 
{\rm (iii)} chart moves of conjugacy type; 
and 
{\rm (iv)} sending $\Gamma$ to $\Gamma'$ by 
orientation preserving diffeomorphisms $h:B\rightarrow B$ which 
satisfies $h(b_0)=b'_0$. 

{\rm (2)} $(f,\Phi)$ is strictly $\mathcal{H}$--isomorphic to $(f',\Phi')$ 
if and only if $\Gamma$ is transformed into $\Gamma'$ by a finite sequence of 
the moves {\rm (i), (ii), (iii),} and {\rm (iv)} provided that 
$h$ is isotopic to the identity relative to a given base point $b_0$. 
\end{prop}

We end this subsection with a definition of fiber sums of Lefschetz fibrations. 
Let $f\co M\rightarrow B$ and $f'\co M'\rightarrow B'$ be Lefschetz fibrations of genus $g$, 
and $\Delta$ and $\Delta'$ the sets of critical values of $f$ and $f'$, respectively. 
Take base points $b_0\in B-\Delta$ and $b_0'\in B'-\Delta'$ for $f$ and $f'$, 
orientation preserving diffeomorphisms $\Phi\co \Sigma_g\rightarrow f^{-1}(b_0)$ 
and $\Phi'\co\Sigma_g\rightarrow f'^{-1}(b'_0)$, 
and small disks $D_0\subset B-\Delta$ and $D'_0\subset B-\Delta'$ near $b_0$ and $b'_0$, 
respectively. 

\begin{defn}\label{sum} 
Let $\Psi\co \Sigma_g\rightarrow \Sigma_g$ be an orientation preserving diffeomorphism 
and $r\co \partial D_0\rightarrow \partial D'_0$ an orientation reversing diffeomorphism. 
The new manifold $M\#_F M'$ obtained by glueing 
$M-f^{-1}({\rm Int}\, D_0)$ and $M'-f'^{-1}({\rm Int}\, D'_0)$ by 
$(\Phi'\circ\Psi\circ\Phi^{-1})\times r$ 
admits a Lefschetz fibration 
$f\#_{\Psi}\, f'\co M\#_F M'\rightarrow B\# B'$ of genus $g$. 
We call $f\#_{\Psi}\, f'$ the {\it fiber sum} of $f$ and $f'$ with respect to $\Psi$. 
Although 
the diffeomorphism type of $M\#_F M'$ and the isomorphism class of $f\#_{\Psi}\, f'$ 
depend on a choice of the diffeomorphism $\Psi$ in general, 
we often abbreviate $f\#_{\Psi}\, f'$ as $f\#\, f'$. 
\end{defn}

Let $\rho\co\pi_1(B-\Delta,b_0)\rightarrow \mathcal{M}_g$ and 
$\rho'\co\pi_1(B'-\Delta',b'_0)\rightarrow \mathcal{M}_g$ be the 
monodromy representations of $f$ and $f'$ with respect to $\Phi$ and $\Phi'$, 
respectively. 
Consider the fiber sum $f\#_{\Psi}\, f'$ of $f$ and $f'$ with respect to $\Psi$ as in 
Definition \ref{sum} and the monodromy representation 
$\tilde{\rho}\co\pi_1(B\# B'-(\Delta\cup\Delta'),b_0)\rightarrow \mathcal{M}_g$ 
with respect to $\Phi$. 

\begin{defn} 
Suppose that $(f,\Phi)$ and $(f',\Phi')$ are hyperelliptic Lefschetz fibrations and 
the isotopy class of $\Psi$ belongs to $\mathcal{H}_g$. 
We call the pair $(f\#_{\Psi}\, f',\Phi)$ the {\it $\mathcal{H}$--fiber sum} of 
$(f,\Phi)$ and $(f',\Phi')$ with respect to $\Psi$. 
The $\mathcal{H}$--fiber sum $(f\#_{\Psi}\, f',\Phi)$ is 
a hyperelliptic Lefschetz fibration because it is easily seen that the image of $\tilde{\rho}$ 
is included in $\mathcal{H}_g$. 
\end{defn}

Let $\Gamma$ and $\Gamma'$ be $\hat{\mathcal{C}}$--charts 
corresponding to $(f,\Phi)$ and $(f',\Phi')$, 
and $D_0$ and $D'_0$ small disks near $b_0$ and $b_0'$ disjoint from 
$\Gamma$ and $\Gamma'$, respectively. 
Connecting $B-{\rm Int}\, D_0$ with $B'-{\rm Int}\, D'_0$ by a tube, 
we have a connected sum $B\# B'$ of $B$ and $B'$. 
Let $w$ be a word in $\hat{\mathcal{X}}\cup\hat{\mathcal{X}}^{-1}$ which represents 
the mapping class of $\Psi$ in $\mathcal{H}_g$. 
Let $\Gamma\#_w\Gamma'$ be the union of $\Gamma$, $\Gamma'$, and 
hoops on the tube representing $w$. 
%(see Figure \ref{fibersum}). 
Then the $\mathcal{H}$--fiber sum $(f\#_{\Psi}\, f',\Phi)$ is described by 
this new chart $\Gamma\#_w\Gamma'$ in $B\# B'$ with base point $b_0$. 
If the word $w$ is trivial, then the chart $\Gamma\#_w\Gamma'$ is denoted also 
by $\Gamma\oplus\Gamma'$, 
which is called a {\it product} of $\Gamma$ and $\Gamma'$.

%%%%%%%%%%%%% Section 4 %%%%%%%%%%%%%%%

\section{Counting Dirac braid relators}

%%%%%%%%%%%%%%%%%%%%%%%%%%%%%%%%%

In this section we define a $\mathbb{Z}_2$--valued invariant for $\mathcal{C}_0$--charts 
in a given surface and prove its invariance under several chart moves. 

We first give a precise definition of the invariant. 
Let $B$ be a connected closed oriented surface and $g$ an integer greater than one. 

\begin{defn}\label{w} 
For a $\mathcal{C}_0$--chart $\Gamma$ in $B$, 
we denote the number modulo $2$ of white vertices of types $r_4^{\pm 1}$ 
included in $\Gamma$ by $w(\Gamma)$. 
(Note that we named $r_4$ the Dirac braid relator in \S \ref{chartdes}.) 
\end{defn}

The value of $w$ is obviously invariant under chart moves of conjugacy type. 
If a $\mathcal{C}_0$--chart $\Gamma$ with base point $b_0$ is sent to 
another $\mathcal{C}_0$--chart $\Gamma'$ with base point $b'_0$ by 
an orientation preserving diffeomorphism $h:B\rightarrow B$ which satisfies $h(b_0)=b'_0$, 
we have $w(\Gamma)=w(\Gamma')$. 

\begin{prop}\label{typew}
The value of $w$ is invariant under chart moves of type W. 
\end{prop}

We need a lemma. 

\begin{lem}\label{fill}
For any $k=1,\ldots ,2g+1$ and 
a $\mathcal{C}_0$--chart depicted in Figure \ref{center}, 
there exists a filling of the box labeled $T_k$ which includes neither black vertices 
nor white vertices of types $r_4^{\pm 1}$. 
%can be filled only with edges and white vertices of 
%types $r_1(i,j)^{\pm 1}$, $r_2(i)^{\pm 1}$, and $r_3^{\pm 1}$. 
(Note that the sequence of edges labeled $1,\ldots ,2g+1$ in Figure \ref{center} 
appears $2g+2$ times on each side of the box.) 

\begin{figure}[ht!]
\labellist
\footnotesize \hair 2pt
\pinlabel $2g+1$ [r] at 0 128
\pinlabel $1$ [r] at 0 108
\pinlabel $2g+1$ [r] at 0 72
\pinlabel $1$ [r] at 0 54
\pinlabel $2g+1$ [r] at 0 43
\pinlabel $1$ [r] at 0 24
\pinlabel $k$ [l] at 48 150
\pinlabel $k$ [l] at 48 5
\pinlabel {\small $T_k$} [l] at 38 77
\pinlabel $2g+1$ [l] at 93 128
\pinlabel $1$ [l] at 93 108
\pinlabel $2g+1$ [l] at 93 72
\pinlabel $1$ [l] at 93 54
\pinlabel $2g+1$ [l] at 93 43
\pinlabel $1$ [l] at 93 24
\endlabellist
\centering
\includegraphics[scale=0.75]{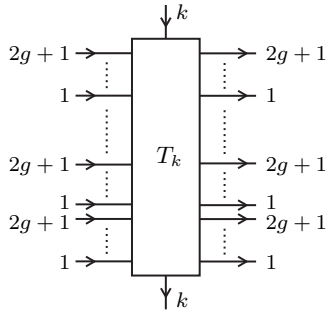}
\caption{Subchart without white vertices of types $r_4^{\pm 1}$}
\label{center}
\end{figure}

\end{lem}

\begin{proof}
Since the element of $B_{2g+2}(S^2)$ represented by 
$(x_1x_2\cdots x_{2g+1})^{2g+2}$ is included in the center of $B_{2g+2}(S^2)$, 
it commutes with each of $x_1,\ldots ,x_{2g+1}$ in $B_{2g+2}(S^2)$. 
Thus the word $[x_k,(x_1x_2\cdots x_{2g+1})^{2g+2}]$ represents the identity element 
of $B_{2g+2}(S^2)$, 
and there exists a finite sequence of words in $\tilde{\mathcal{X}}\cup\tilde{\mathcal{X}}^{-1}$ 
starting from the word $[x_k,(x_1x_2\cdots x_{2g+1})^{2g+2}]$ to the empty word 
such that each word is related to the previous one by one of the following 
transformations: 
(i) insertion or deletion of a trivial relator $x^{\varepsilon}x^{-\varepsilon}$ 
for $x\in\tilde{\mathcal{X}}$ and $\varepsilon\in\{+1,-1\}$; 
(ii) insertion of $\tilde{r}^{\varepsilon}$ for 
$\tilde{r}\in\tilde{\mathcal{R}}$ and $\varepsilon\in\{+1,-1\}$. 
We first consider a $\tilde{\mathcal{C}}$--chart depicted in Figure \ref{center}. 
For each transformation (i) (resp. (ii)), we create an edge labeled $x$ 
(resp. a vertex of type $\tilde{r}^{\varepsilon}$). 
Repeating such creations, 
we can fill the box labeled $T_k$ with edges and white vertices of types 
$\tilde{r}_1(i,j)^{\pm 1}$, $\tilde{r}_2(i)^{\pm 1}$, and $\tilde{r}_3^{\pm 1}$. 
(A similar argument for charts without white vertices of type $\tilde{r}_3^{\pm 1}$ 
can be found in \cite[Lemma 15]{Kamada1992} and \cite[\S 18.2]{Kamada2002}.) 
Changing all labels $x_i$ of edges into $\xi_i$, we obtain a $\mathcal{C}_0$--chart 
which consists of edges and white vertices of types 
$r_1(i,j)^{\pm 1}$, $r_2(i)^{\pm 1}$, and $r_3^{\pm 1}$. 
\end{proof}

\begin{proof}[of Proposition \ref{typew}] 
It suffices to show that $w(\Gamma)=0$ for every $\mathcal{C}_0$--chart $\Gamma$ 
in $S^2$ without black vertices. 
Let $\Gamma$ be such a $\mathcal{C}_0$--chart in $S^2$. 
We consider a chart move of type W depicted in Figure \ref{movesD}. 
Suppose that the box labeled $T_k$ is filled with edges and white vertices of 
types $r_1(i,j)^{\pm 1}$, $r_2(i)^{\pm 1}$, and $r_3^{\pm 1}$ by Lemma \ref{fill}. 

\begin{figure}[ht!]
\labellist
\footnotesize \hair 2pt
\pinlabel $2g+1$ [r] at 0 138
\pinlabel $1$ [r] at 0 119
\pinlabel $2g+1$ [r] at 0 83
\pinlabel $1$ [r] at 0 64
\pinlabel $2g+1$ [r] at 0 54
\pinlabel $1$ [r] at 0 35
\pinlabel $r_4$ [r] at 46 88
\pinlabel $k$ [l] at 76 159
\pinlabel $2g+1$ [r] at 224 138
\pinlabel $1$ [r] at 224 119
\pinlabel $2g+1$ [r] at 224 83
\pinlabel $1$ [r] at 224 64
\pinlabel $2g+1$ [r] at 224 54
\pinlabel $1$ [r] at 224 35
\pinlabel $r_4$ [r] at 332 88
\pinlabel $k$ [l] at 273 159
\pinlabel $k$ [l] at 273 13
\pinlabel {\small $T_k$} [r] at 279 88
\endlabellist
\centering
\includegraphics[scale=0.75]{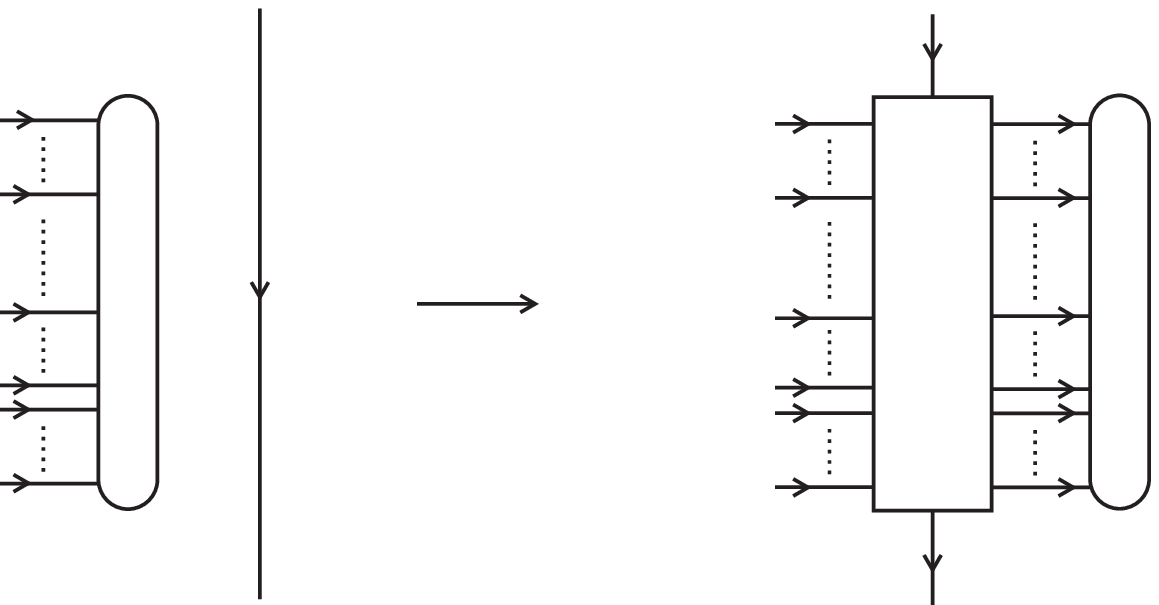}
\caption{Passing through an edge}
\label{movesD}
\end{figure}

Using the sequence of chart moves repeatedly, we can make 
a $\mathcal{C}_0$--chart $\Gamma_1$ 
depicted in Figure \ref{chartA} (a) from $\Gamma$, 
where the box labeled $\Theta_1$ is filled with edges and white vertices of 
types $r_1(i,j)^{\pm 1}$, $r_2(i)^{\pm 1}$, and $r_3^{\pm 1}$. 
The number of white vertices of type $r_4$ and that of white vertices of type $r_4^{-1}$ 
included in $\Gamma_1$ are the same as those for $\Gamma$ respectively. 

\begin{figure}[ht!]
\labellist
\footnotesize \hair 2pt
\pinlabel {\small $\Theta_1$} [r] at 83 53
\pinlabel $r_4^{\pm 1}$ [b] at 30 0
\pinlabel $r_4^{\pm 1}$ [b] at 114 0
\pinlabel (a) [t] at 72 0
\pinlabel {\small $\Theta_2$} [r] at 240 53
\pinlabel $r_4^{\varepsilon}$ [b] at 188 0
\pinlabel $r_4^{\varepsilon}$ [b] at 275 0
\pinlabel (b) [t] at 230 0
\pinlabel {\small $\Theta_2$} [r] at 400 53
\pinlabel $w$ [b] at 470 26
\pinlabel (c) [t] at 390 0
\endlabellist
\centering
\includegraphics[scale=0.75]{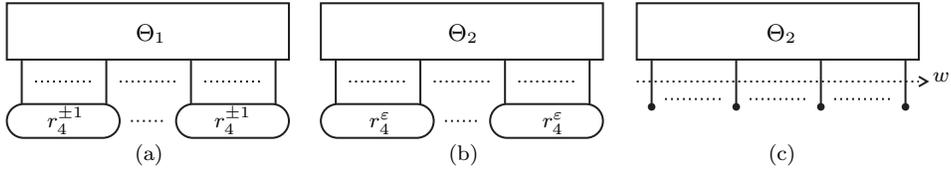}
\vspace{3mm}
\caption{Charts $\Gamma_1$, $\Gamma_2$, and $\Gamma_3$}
\label{chartA}
\end{figure}

We then apply deaths of pairs of white vertices of types $r_4$ and $r_4^{-1}$ as in 
Figure \ref{movesA} (c) and sequences of chart moves of type W as in Figure \ref{movesD} 
to $\Gamma_1$ repeatedly 
to obtain a $\mathcal{C}_0$--chart $\Gamma_2$ depicted in Figure \ref{chartA} (b), 
where the box labeled $\Theta_2$ is filled with edges and white vertices of 
types $r_1(i,j)^{\pm 1}$, $r_2(i)^{\pm 1}$, and $r_3^{\pm 1}$, 
and $\varepsilon$ is equal to either $+1$ or $-1$. 
Let $n$ be the number of white vertices of type $r_4^{\varepsilon}$ included in $\Gamma_2$. 
We replace all the white vertices of type $r_4^{\varepsilon}$ included in $\Gamma_2$ 
with black vertices of types $\ell_0(i)^{-\varepsilon}$ to obtain 
a $\mathcal{C}_0$--chart $\Gamma_3$ depicted in Figure \ref{chartA} (c). 
The intersection word $w$ of the dashed arrow with respect to $\Gamma_3$ 
in Figure \ref{chartA} (c) is equal to 
$(\xi_1\xi_2\cdots\xi_{2g+1})^{\varepsilon (2g+2)n}$. 
Changing all labels $\xi_i$ of edges of $\Gamma_3$ into $x_i$, 
we obtain a $\tilde{\mathcal{C}}$--chart $\tilde{\Gamma}_3$ 
because $\Gamma_3$ does not contain white vertices of types $r_4^{\pm 1}$. 
The intersection word $w$ of the dashed arrow with respect to $\tilde{\Gamma}_3$ 
in Figure \ref{chartA} (c) is equal to 
$(x_1x_2\cdots x_{2g+1})^{\varepsilon (2g+2)n}$, which represents the identity 
element of $B_{2g+2}(S^2)$. 
Since the element of $B_{2g+2}(S^2)$ represented by 
$(x_1x_2\cdots x_{2g+1})^{2g+2}$ is of order $2$ (see \cite[\S 9.1.4]{FM2011}), 
$n$ must be even. Therefore we have 
\[
w(\Gamma)=w(\Gamma_1)=w(\Gamma_2)=w(\Gamma_3)=0
\]
and this completes the proof. 
\end{proof}

\begin{prop}\label{transition}
Assume that $g$ is odd. 
Then the value of $w$ is invariant under chart moves of transition. 
\end{prop}

We divide chart moves of transition for $\mathcal{C}_0$--charts into two classes. 

\begin{defn} 
A chart move of transition as in Figure \ref{movesB} is called 
an ${\rm L}_0$--{\it move} (resp. ${\rm L}_+$--{\it move}) if $\mathcal{C}=\mathcal{C}_0$, 
$w\in\mathcal{X}_0\cup\mathcal{X}_0^{-1}$, and 
$s,s'\in\{\ell_0(i)^{\pm 1}\, |\, i=1,\ldots ,2g+1\}$ 
(resp. $s,s'\in\{\ell_h^{\pm 1}\, |\, h=1,\ldots ,[g/2]\}$). 
If $s=\xi_k^{\varepsilon}$, $s'=\xi_{\ell}^{\varepsilon}$ 
(resp. $s=s'=(\xi_1\xi_2\cdots\xi_{2h})^{\varepsilon(4h+2)}$), 
and $\varepsilon\in\{+1,-1\}$, 
then the label $T$ on the box in Figure \ref{movesB} is also denoted by 
$T_1(k,\ell; \varepsilon)$ (resp. $T_2(h;\varepsilon)$). 
Every chart move of transition for $\mathcal{C}_0$--charts is either an ${\rm L}_0$--move or 
an ${\rm L}_+$--move. 
\end{defn}

We first show a lemma for ${\rm L}_0$--moves. 

\begin{lem}\label{t1} 
For any $k,\ell=1,\ldots ,2g+1$ and $\varepsilon\in\{+1,-1\}$, 
there exists a filling of 
the box labeled $T=T_1(k,\ell;\varepsilon)$ in Figure \ref{movesB} 
without black vertices such that the number of white vertices of type $r_4^{\pm 1}$ is even, 
i.e., the filling consists of edges, white vertices of 
types $r_1(i,j)^{\pm 1}$, $r_2(i)^{\pm 1}$, $r_3^{\pm 1}$, 
and an even number of white vertices of types $r_4^{\pm 1}$. 
\end{lem}

\begin{proof} 
We assume that $\varepsilon=+1$ for simplicity. 
Let $\varphi$ and $\delta_i$ be elements of $\mathcal{M}_{0,2g+2}$ 
represented by $w$ and $\xi_i$, respectively. 
Since the intersection word of the boundary of the box labeled $T=T_1(k,\ell;+1)$ 
with respect to the $\mathcal{C}_0$--chart in Figure \ref{movesB} 
is $wsw^{-1}s'^{-1}=w\xi_kw^{-1}\xi_{\ell}^{-1}$, we have a relation 
$\varphi\delta_k\varphi^{-1}\delta_{\ell}^{-1}=1$ in $\mathcal{M}_{0,2g+2}$. 

Case 1: Suppose that $k=\ell$. 
We consider a hyperelliptic involution $I$ on $\Sigma_g$ and 
think of $S^2$ as the quotient of $\Sigma_g$ by the action of $I$. 
The image of the simple closed curve $c_i$ in Figure \ref{curves1} 
under the double branched covering $\Sigma_g\rightarrow S^2$ 
is a simple arc $a_i$ on $S^2$ depicted in Figure \ref{arcs}, 
and a right-handed half twist $D_i$ about $a_i$ represents the mapping class $\delta_i$. 
We also consider an additional arc $a_{2g+2}$ on $S^2$ as in Figure \ref{arcs}, 
a right-handed half twist $D_{2g+2}$ about $a_{2g+2}$, 
and the mapping class $\delta_{2g+2}$ of $D_{2g+2}$. 
We think of the indices $i$ of $a_i$, $D_i$, and $\delta_i$ as integers modulo $2g+2$. 
For example, $a_{2g+3}=a_1$. 

\begin{figure}[ht!]
\labellist
\small \hair 2pt
\pinlabel $a_1$ [t] at 33 6
\pinlabel $a_2$ [t] at 66 6
\pinlabel $a_i$ [t] at 162 6
\pinlabel $a_{2g+1}$ [t] at 251 6
\pinlabel $a_{2g+2}$ [t] at 141 46
\endlabellist
\centering
\includegraphics[scale=0.75]{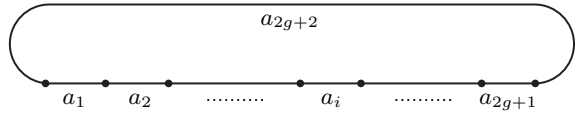}
\vspace{2mm}
\caption{Arcs on $S^2$}
\label{arcs}
\end{figure}

The relation $\varphi\delta_k\varphi^{-1}=\delta_k$ implies that 
$(a_k)F$ is isotopic to $a_k$ by Lemma 4.1 of \cite{KM2000}, 
where $F$ is an orientation preserving diffeomorphism on $S^2$
representing $\varphi$. 
Since $(a_k)D_k=a_k$ and $D_k$ reverses the orientation of $a_k$, 
we can assume that either $F$ or $FD_k$ fixes $a_k$ pointwise. 
Cutting $S^2$ along $a_k$, we obtain an orientation preserving diffeomorphism 
on a $2$--disk which fixes the boundary pointwise. 
The mapping class group $\mathcal{M}_{0,2g}^1$ of the $2$--disk with $2g$ marked points 
is generated by the half twists about arcs corresponding to $a_1,\ldots ,a_{2g+2}$ 
except $a_{k-1}$ and $a_{k+1}$ (see \cite[\S 9.1.3]{FM2011}). 
Hence $\varphi$ can be factorized 
as a product of $\delta_1^{\pm 1},\ldots ,\delta_{2g+2}^{\pm 1}$ except 
$\delta_{k-1}^{\pm 1}$ and $\delta_{k+1}^{\pm 1}$, 
and it can be represented by a word $w'$ in 
$\xi_1^{\pm 1},\ldots ,\xi_{2g+2}^{\pm 1}$ except 
$\xi_{k-1}^{\pm 1}$ and $\xi_{k+1}^{\pm 1}$, 
where we put 
$\xi_{2g+2}:=\xi_1\xi_2\cdots\xi_{2g}\xi_{2g+1}\xi_{2g}^{-1}\cdots\xi_2^{-1}\xi_1^{-1}$. 
It is easily checked that the relation 
$\delta_{2g+2}=\delta_1\delta_2\cdots\delta_{2g}\delta_{2g+1}
\delta_{2g}^{-1}\cdots\delta_2^{-1}\delta_1^{-1}$ holds in $\mathcal{M}_{0,2g+2}$. 

We divide the box labeled $T_1(k,k;+1)$ into three boxes labeled 
$T_1(k)$, $\Theta_1(k)$, and $\Theta_1(k)^*$ as in Figure \ref{chartB}. 

\begin{figure}[ht!]
\labellist
\footnotesize \hair 2pt
\pinlabel $k$ [r] at 0 52
\pinlabel $k$ [l] at 133 52
\pinlabel $w$ [l] at 116 101
\pinlabel $w$ [l] at 116 3
\pinlabel $T_1(k,k;+1)$ [r] at 98 51
\pinlabel $=$ [l] at 160 50
\pinlabel $k$ [r] at 202 52
\pinlabel $k$ [l] at 335 52
\pinlabel $w$ [l] at 318 101
\pinlabel $w$ [l] at 318 3
\pinlabel $w'$ [l] at 325 76
\pinlabel $w'$ [l] at 325 30
\pinlabel $\Theta_1(k)^*$ [r] at 292 85
\pinlabel $\Theta_1(k)$ [r] at 287 16
\pinlabel $T_1(k)$ [r] at 287 51
\endlabellist
\centering
\includegraphics[scale=0.8]{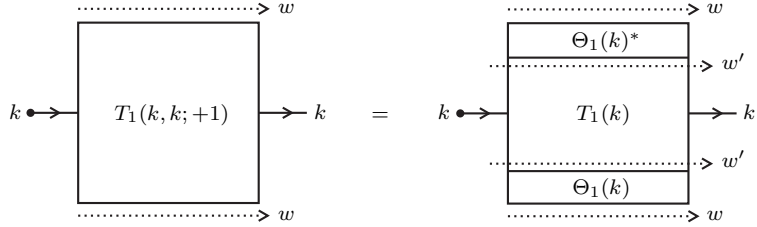}
\caption{Division of a box}
\label{chartB}
\end{figure}

The box labeled $\Theta_1(k)$ can be filled with edges and white vertices 
because both $w$ and $w'$ represent the same mapping class $\varphi$. 
The box labeled $\Theta_1(k)^*$ can be filled with the mirror image of the subchart 
filling the box labeled $\Theta_1(k)$ with edges orientation reversed. 
Since $w'$ is a word in $\xi_1^{\pm 1},\ldots ,\xi_{2g+1}^{\pm 1}, 
\xi_{2g+2}^{\pm 1} (=
\xi_1\xi_2\cdots\xi_{2g}\xi_{2g+1}\xi_{2g}^{-1}\cdots\xi_2^{-1}\xi_1^{-1})$ except 
$\xi_{k-1}^{\pm 1}$ and $\xi_{k+1}^{\pm 1}$, 
the box labeled $T_1(k)$ can be filled with copies of three kinds of 
subcharts depicted in Figure \ref{chartC}. 
Note that the subcharts depicted in Figure \ref{chartC} correspond to $\xi_k$, 
$\xi_i\, (i\ne k-1,k,k+1,2g+2)$, and $\xi_{2g+2}$, respectively. 
We also need subcharts corresponding to $\xi_k^{-1}$, 
$\xi_i^{-1}\, (i\ne k-1,k,k+1,2g+2)$, and $\xi_{2g+2}^{-1}$. 

\begin{figure}[ht!]
\labellist
\footnotesize \hair 2pt
\pinlabel $k$ [r] at 0 43
\pinlabel $k$ [l] at 36 43
\pinlabel $k$ [b] at 18 86
\pinlabel $k$ [t] at 18 0
\pinlabel $k$ [r] at 67 43
\pinlabel $k$ [l] at 102 43
\pinlabel $i$ [b] at 84 86
\pinlabel $i$ [t] at 84 0
\pinlabel $k$ [r] at 136 43
\pinlabel $k$ [l] at 342 43
\pinlabel $1$ [b] at 151 86
\pinlabel $2$ [b] at 166 86
\pinlabel $k$--$1$ [b] at 186 86
\pinlabel $k$ [b] at 200 86
\pinlabel $2g$ [b] at 225 86
%\pinlabel $2g+1$ [b] at 240 86
\pinlabel $2g$ [b] at 253 86
\pinlabel $k$ [b] at 276 86
\pinlabel $k$--$1$ [b] at 289 86
\pinlabel $2$ [b] at 314 86
\pinlabel $1$ [b] at 328 86
\pinlabel $1$ [t] at 151 0
\pinlabel $2$ [t] at 166 0
\pinlabel $k$--$1$ [t] at 186 0
\pinlabel $k$ [t] at 200 0
%\pinlabel $2g$ [t] at 225 0
\pinlabel $2g$+$1$ [t] at 240 0
%\pinlabel $2g$ [t] at 253 0
\pinlabel $k$ [t] at 276 0
\pinlabel $k$--$1$ [t] at 289 0
\pinlabel $2$ [t] at 314 0
\pinlabel $1$ [t] at 328 0
\pinlabel $k$--$1$ [b] at 212 47
\pinlabel $k$--$1$ [b] at 264 47
\endlabellist
\centering
\includegraphics[scale=0.8]{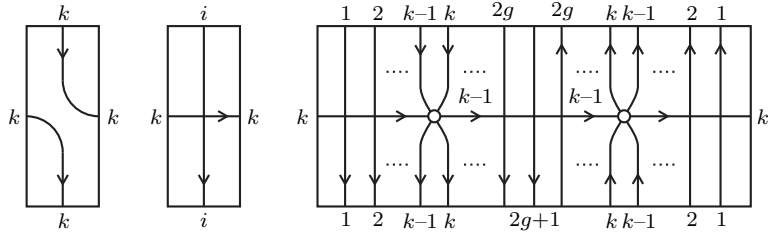}
\vspace{3mm}
\caption{Three kinds of subcharts}
\label{chartC}
\end{figure}

The box labeled $T_1(k)$ is filled with edges and 
white vertices of types $r_1(i,j)^{\pm 1}$ and $r_2(i)^{\pm 1}$. 
The number of white vertices of type $r_4^{\pm 1}$ included in the box labeled $\Theta_1(k)^*$ 
is equal to that for the box labeled $\Theta_1(k)$. 
Therefore the box labeled $T_1(k,k;+1)$ is filled with edges, white vertices of 
types $r_1(i,j)^{\pm 1}$, $r_2(i)^{\pm 1}$, $r_3^{\pm 1}$, 
and an even number of white vertices of types $r_4^{\pm 1}$. 

Case 2: Suppose that $k\ne\ell$. 
We assume that $k<\ell$ for simplicity. 
We put $v:=(\xi_{k+1}\xi_k)(\xi_{k+2}\xi_{k+1})\cdots (\xi_{\ell}\xi_{\ell-1})$. 
Since $v^{-1}\xi_kv$ and $\xi_{\ell}$ represent the same element of $\mathcal{M}_{0,2g+2}$, 
and $w\xi_kw^{-1}\xi_{\ell}^{-1}$ represents the identity element of $\mathcal{M}_{0,2g+2}$, 
$vw\xi_k(vw)^{-1}\xi_k^{-1}$ also represents the identity element of $\mathcal{M}_{0,2g+2}$. 
We fill the box labeled $T_1(k,\ell;+1)$ 
with edges, white vertices of type $r_2(i)^{\pm 1}$, 
and a small box labeled $T_1(k,k;+1)'$ as in Figure \ref{chartD}. 

\begin{figure}[ht!]
\labellist
\footnotesize \hair 2pt
\pinlabel $\ell$ [r] at 0 89
\pinlabel $k$ [l] at 243 89
\pinlabel $k$ [b] at 83 92
\pinlabel $v$ [b] at 158 111
\pinlabel $w$ [b] at 212 111
\pinlabel $v$ [t] at 158 66
\pinlabel $w$ [t] at 212 66
\pinlabel $T_1(k,k;+1)'$ [r] at 188 88
\endlabellist
\centering
\includegraphics[scale=0.8]{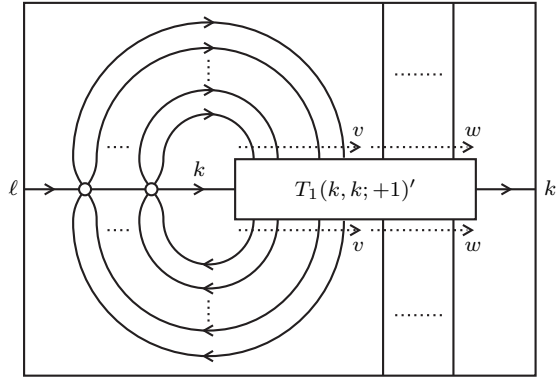}
\caption{Inside of the box labeled $T_1(k,\ell;+1)$}
\label{chartD}
\end{figure}

By the proof of Case 1, 
the small box labeled $T_1(k,k;+1)'$ is filled with edges, white vertices of 
types $r_1(i,j)^{\pm 1}$, $r_2(i)^{\pm 1}$, $r_3^{\pm 1}$, 
and an even number of white vertices of types $r_4^{\pm 1}$, 
so is the box labeled $T_1(k,\ell;+1)$. 
\end{proof}

We next show a lemma for ${\rm L}_+$--moves. 

\begin{lem}\label{t2} 
Assume that $g$ is odd. 
For any $h=1,\ldots ,[g/2]$ and $\varepsilon\in\{+1,-1\}$, 
there exists a filling of 
the box labeled $T=T_2(h;\varepsilon)$ in Figure \ref{movesB} 
without black vertices such that the number of white vertices of type $r_4^{\pm 1}$ is even, 
i.e., the filling consists of edges, white vertices of 
types $r_1(i,j)^{\pm 1}$, $r_2(i)^{\pm 1}$, $r_3^{\pm 1}$, 
and an even number of white vertices of types $r_4^{\pm 1}$. 
\end{lem}

\begin{proof} 
We assume that $\varepsilon=+1$ for simplicity. 
Let $\varphi$ and $\tau_h$ be elements of $\mathcal{M}_{0,2g+2}$ 
represented by $w$ and $(\xi_1\xi_2\cdots\xi_{2h})^{4h+2}$, respectively. 
Since the intersection word of the boundary of the box labeled $T=T_2(h;+1)$ 
with respect to the $\mathcal{C}_0$--chart in Figure \ref{movesB} 
is $wsw^{-1}s'^{-1}=
w(\xi_1\xi_2\cdots\xi_{2h})^{4h+2}w^{-1}(\xi_1\xi_2\cdots\xi_{2h})^{-(4h+2)}$, 
we have a relation 
$\varphi\tau_h\varphi^{-1}\tau_h^{-1}=1$ in $\mathcal{M}_{0,2g+2}$. 

We consider a hyperelliptic involution $I$ on $\Sigma_g$ and 
think of $S^2$ as the quotient of $\Sigma_g$ by the action of $I$. 
The image of the simple closed curve $s_h$ in Figure \ref{curves1} 
under the double branched covering $\Sigma_g\rightarrow S^2$ 
is a simple closed curve $b_h$ on $S^2$ depicted in Figure \ref{scc}, 
and the square of a right-handed Dehn twist $T_h$ about $b_h$ 
represents the mapping class $\tau_h$. 
Note that $a_1,\ldots ,a_{2g+1}$ in Figure \ref{scc} 
are the same arcs as those depicted in Figure \ref{arcs}. 

\begin{figure}[ht!]
\labellist
\small \hair 2pt
\pinlabel $a_1$ [t] at 41 26
\pinlabel $a_{2h}$ [t] at 99 26
\pinlabel $a_{2h+2}$ [t] at 185 26
\pinlabel $a_{2g+1}$ [t] at 244 26
\pinlabel $b_h$ [l] at 137 52
\endlabellist
\centering
\includegraphics[scale=0.75]{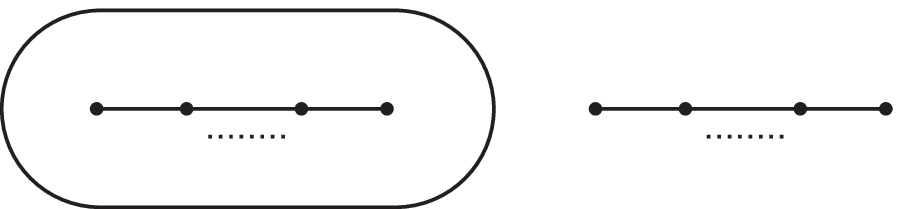}
\caption{Simple closed curve $b_h$ on $S^2$}
\label{scc}
\end{figure}

The relation $\varphi\tau_h\varphi^{-1}=\tau_h$ implies that 
$(b_h)F$ is isotopic to $b_h$ (see \cite[Fact 3.6]{FM2011} 
and \cite[Proposition 3.6]{PR2000}), 
where $F$ is an orientation preserving diffeomorphism on $S^2$
representing $\varphi$. 
Since $g$ is odd, two components of $S^2-b_h$ include different numbers of 
branch points. Thus we can assume that $F$ fixes $b_h$ pointwise. 
Cutting $S^2$ along $b_h$, we obtain orientation preserving diffeomorphisms 
on two $2$--disks which fix the boundary pointwise. 
The mapping class group $\mathcal{M}_{0,2h+1}^1$ (resp. $\mathcal{M}_{0,2g-2h+1}^1$) 
of the $2$--disk with $2h+1$ (resp. $2g-2h+1$) marked points 
is generated by the half twists about arcs corresponding to $a_1,\ldots ,a_{2h}$ 
(resp. $a_{2h+2},\ldots ,a_{2g+1}$) (see \cite[\S 9.1.3]{FM2011}). 
Hence $\varphi$ can be factorized 
as a product of $\delta_1^{\pm 1},\ldots ,\delta_{2h}^{\pm 1}$, $\delta_{2h+2}^{\pm 1},\ldots , 
\delta_{2g+1}^{\pm 1}$, and it can be represented by a word $w'$ in 
$\xi_1^{\pm 1},\ldots ,\xi_{2h}^{\pm 1}$, $\xi_{2h+2}^{\pm 1},\ldots ,\xi_{2g+1}^{\pm 1}$. 

We divide the box labeled $T_2(h;+1)$ into three boxes labeled 
$T_2(h)$, $\Theta_2(h)$, and $\Theta_2(h)^*$ as in Figure \ref{chartE}. 

\begin{figure}[ht!]
\labellist
\footnotesize \hair 2pt
\pinlabel $\ell_2(h)$ [r] at 0 52
%\pinlabel $k$ [l] at 133 52
\pinlabel $w$ [l] at 130 101
\pinlabel $w$ [l] at 130 3
\pinlabel $T_2(h;+1)$ [r] at 110 51
\pinlabel $=$ [l] at 157 50
\pinlabel $\ell_2(h)$ [r] at 202 52
%\pinlabel $k$ [l] at 335 52
\pinlabel $w$ [l] at 333 101
\pinlabel $w$ [l] at 333 3
\pinlabel $w'$ [l] at 340 76
\pinlabel $w'$ [l] at 340 30
\pinlabel $\Theta_2(h)^*$ [r] at 308 85
\pinlabel $\Theta_2(h)$ [r] at 303 16
\pinlabel $T_2(h)$ [r] at 303 51
\endlabellist
\centering
\includegraphics[scale=0.8]{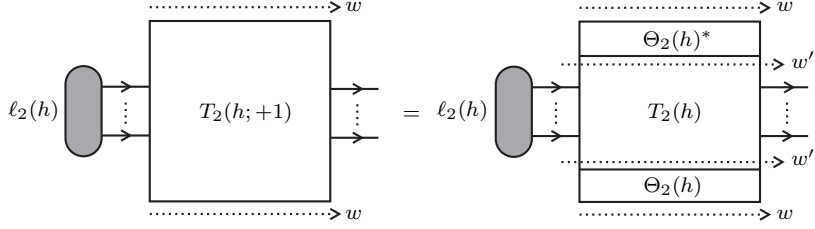}
\caption{Division of a box}
\label{chartE}
\end{figure}

The box labeled $\Theta_2(h)$ can be filled with edges and white vertices 
because both $w$ and $w'$ represent the same mapping class $\varphi$. 
The box labeled $\Theta_2(h)^*$ can be filled with the mirror image of the subchart 
filling the box labeled $\Theta_2(h)$ with edges orientation reversed. 
Since $w'$ is a word in 
$\xi_1^{\pm 1},\ldots ,\xi_{2h}^{\pm 1}$, $\xi_{2h+2}^{\pm 1},\ldots ,\xi_{2g+1}^{\pm 1}$, 
the box labeled $T_2(h)$ can be filled with copies of two kinds of 
subcharts depicted in Figure \ref{chartF}. 
Note that the subcharts depicted in Figure \ref{chartF} (a) and (b) correspond to 
$\xi_i\, (i=1,\ldots ,2h)$ and $\xi_j\, (j=2h+2,\ldots ,2g+1)$, respectively. 
The boxes labeled $\Omega_k\, (k=2,\ldots ,2h)$ and $\Omega_1$ in Figure \ref{chartF} (a) 
are filled with the subcharts depicted in Figure \ref{chartF} (c) and Figure \ref{chartG}, 
respectively. We also need subcharts corresponding to 
$\xi_i^{-1}\, (i=1,\ldots ,2h)$ and $\xi_j^{-1}\, (j=2h+2,\ldots ,2g+1)$. 

\begin{figure}[ht!]
\labellist
\footnotesize \hair 2pt
\pinlabel $\Omega_{i+1}$ [b] at 22 119
\pinlabel $\Omega_{2h}$ [b] at 22 78
\pinlabel $\Omega_1$ [b] at 22 60
\pinlabel $\Omega_2$ [b] at 22 44
\pinlabel $\Omega_i$ [b] at 22 0
\pinlabel (a) [t] at 22 -13
\pinlabel $2h$ [r] at 86 120
\pinlabel $1$ [r] at 86 99
\pinlabel $2h$ [r] at 86 64
\pinlabel $1$ [r] at 86 45
\pinlabel $2h$ [r] at 86 36
\pinlabel $1$ [r] at 86 16
\pinlabel $j$ [b] at 109 137
\pinlabel $j$ [t] at 109 0
\pinlabel $2h$ [l] at 130 120
\pinlabel $1$ [l] at 130 99
\pinlabel $2h$ [l] at 130 64
\pinlabel $1$ [l] at 130 45
\pinlabel $2h$ [l] at 130 36
\pinlabel $1$ [l] at 130 16
\pinlabel (b) [t] at 109 -13
\pinlabel $\Omega_k$ [b] at 200 61
\pinlabel $=$ [b] at 234 63
\pinlabel $2h$ [r] at 267 108
\pinlabel $k$+$1$ [r] at 267 87
\pinlabel $k$ [r] at 267 74
\pinlabel $k$--$1$ [r] at 267 60
\pinlabel $k$--$2$ [r] at 267 48
\pinlabel $1$ [r] at 267 22
\pinlabel $k$--$1$ [b] at 311 125
\pinlabel $k$ [t] at 311 9
\pinlabel $2h$ [l] at 354 108
\pinlabel $k$+$1$ [l] at 354 87
\pinlabel $k$ [l] at 354 74
\pinlabel $k$--$1$ [l] at 354 60
\pinlabel $k$--$2$ [l] at 354 48
\pinlabel $1$ [l] at 354 22
\pinlabel (c) [t] at 234 -13
\endlabellist
\centering
\includegraphics[scale=0.8]{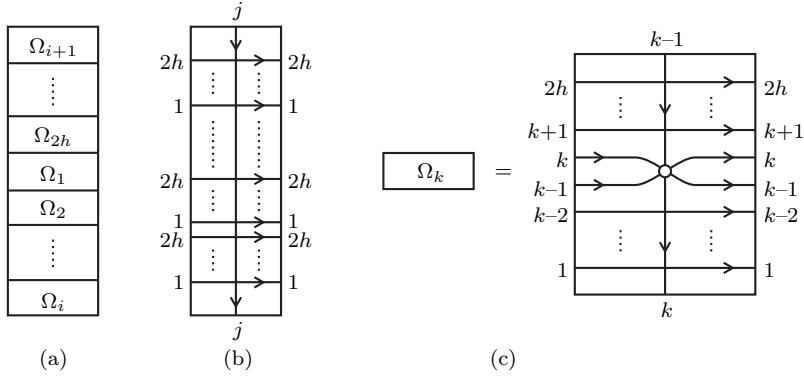}
\vspace{8mm}
\caption{Two kinds of subcharts}
\label{chartF}
\end{figure}

\begin{figure}[ht!]
\labellist
\footnotesize \hair 2pt
\pinlabel $\Omega_1$ [b] at 24 95
\pinlabel $=$ [b] at 65 98
\pinlabel $2h$ [r] at 110 182
\pinlabel $2h$--$1$ [r] at 110 170
\pinlabel $2h$--$2$ [r] at 110 159
\pinlabel $2h$--$3$ [r] at 110 147
\pinlabel $3$ [r] at 110 125
\pinlabel $2$ [r] at 110 114
\pinlabel $1$ [r] at 110 102
\pinlabel $2h$ [r] at 110 90
\pinlabel $2h$--$1$ [r] at 110 78
\pinlabel $2h$--$2$ [r] at 110 67
\pinlabel $4$ [r] at 110 45
\pinlabel $3$ [r] at 110 34
\pinlabel $2$ [r] at 110 23
\pinlabel $1$ [r] at 110 11
\pinlabel $2h$ [b] at 203 194
%\pinlabel $2h$--$1$ [l] at 197 141
\pinlabel {\tiny $2h$--$2$} [l] at 205 131
\pinlabel $3$ [l] at 205 63
\pinlabel $2$ [l] at 205 40
\pinlabel $1$ [t] at 203 0
\pinlabel $2h$ [l] at 296 182
\pinlabel $2h$--$1$ [l] at 296 170
\pinlabel $2h$--$2$ [l] at 296 159
\pinlabel $2h$--$3$ [l] at 296 147
\pinlabel $3$ [l] at 296 125
\pinlabel $2$ [l] at 296 114
\pinlabel $1$ [l] at 296 102
\pinlabel $2h$ [l] at 296 90
\pinlabel $2h$--$1$ [l] at 296 78
\pinlabel $2h$--$2$ [l] at 296 67
\pinlabel $4$ [l] at 296 45
\pinlabel $3$ [l] at 296 34
\pinlabel $2$ [l] at 296 23
\pinlabel $1$ [l] at 296 11
\endlabellist
\centering
\includegraphics[scale=0.8]{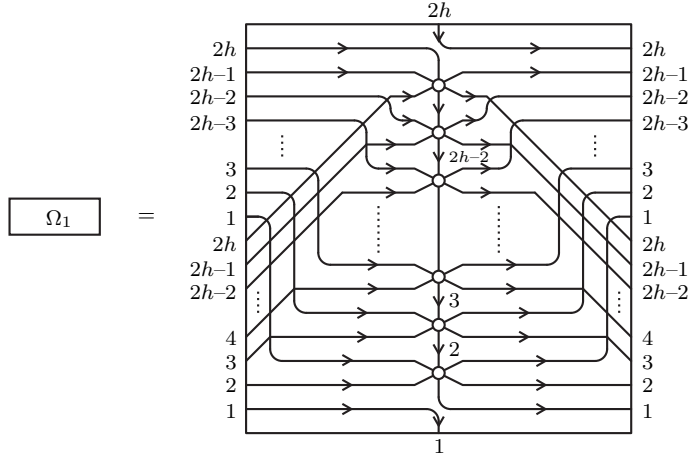}
\bigskip
\caption{Inside of the box labeled $\Omega_1$}
\label{chartG}
\end{figure}

The box labeled $T_2(h)$ is filled with edges and 
white vertices of types $r_1(i,j)^{\pm 1}$ and $r_2(i)^{\pm 1}$. 
The number of white vertices of type $r_4^{\pm 1}$ included in the box labeled $\Theta_2(h)^*$ 
is equal to that for the box labeled $\Theta_2(h)$. 
Therefore the box labeled $T_2(h;+1)$ is filled with edges, white vertices of 
types $r_1(i,j)^{\pm 1}$, $r_2(i)^{\pm 1}$, $r_3^{\pm 1}$, 
and an even number of white vertices of types $r_4^{\pm 1}$. 
\end{proof}

We now prove the invariance of $w$ under chart moves of transition. 

\begin{proof}[of Proposition \ref{transition}] 
Let $\Gamma$ be a $\mathcal{C}_0$--chart $\Gamma$ in $B$ 
and $\Gamma'$ a $\mathcal{C}_0$--chart in $B$ obtained from $\Gamma$ 
by a chart move of transition. 
Suppose that the chart move of transition is an ${\rm L}_0$--move. 
Two $\mathcal{C}_0$--charts $\Gamma$ and $\Gamma'$ are related by a 
chart move as in Figure \ref{movesB} and the box is labeled $T=T_1(k,\ell)$ for some 
$k,\ell=1,\ldots ,2g+1$. 
Let $\Gamma_1$ be a $\mathcal{C}_0$--chart in $B$ obtained from $\Gamma'$ 
by replacing the subchart inside the box labeled $T_1(k,\ell)$ with 
a subchart satisfying the condition given in Lemma \ref{t1}. 
By virtue of Proposition \ref{typew} and Lemma \ref{t1}, 
we have $w(\Gamma)=w(\Gamma_1)=w(\Gamma')$. 
It follows from a similar argument together with Proposition \ref{typew} and 
Lemma \ref{t2} that $w(\Gamma)=w(\Gamma')$ if $g$ is odd and 
$\Gamma'$ is obtained from $\Gamma$ by an ${\rm L}_+$--move. 
Thus we have proved the proposition. 
\end{proof}

%%%%%%%%%%%%% Section 5 %%%%%%%%%%%%%%%

\section{Stable classification}\label{section5}

%%%%%%%%%%%%%%%%%%%%%%%%%%%%%%%%%

In this section we define a $\mathbb{Z}_2$--valued invariant of 
hyperelliptic Lefschetz fibrations of odd genus and 
show a stable classification theorem for hyperelliptic Lefschetz fibrations. 

We first give a definition of the invariant. 
Let $B$ be a connected closed oriented surface and $g$ an integer greater than one. 

\begin{defn}\label{w2} 
Let $(f,\Phi)$ be a hyperelliptic Lefschetz fibration of genus $g$ over $B$ and 
$\rho:\pi_1(B-\Delta,b_0)\rightarrow\mathcal{H}_g$ the monodromy representation 
with respect to $\Phi$. 
Let $\Gamma$ be a $\mathcal{C}_0$--chart in $B$ 
such that the homomorphism determined by $\Gamma$ is equal to 
$\rho_0:=\pi\circ \rho:\pi_1(B-\Delta,b_0)\rightarrow\mathcal{M}_{0,2g+2}$. 
Suppose that $g$ is odd. 
We set $w(f,\Phi):=w(\Gamma)$, where $w(\Gamma)$ is an element of $\mathbb{Z}_2$ 
defined in Definition \ref{w}. 
\end{defn}

\begin{rem}
If we have a $\hat{\mathcal{C}}$--chart $\hat{\Gamma}$ corresponding to $(f,\Phi)$, 
a $\mathcal{C}_0$--chart $\Gamma$ above is obtained from $\hat{\Gamma}$ 
by changing the labels $\zeta_i$ of all edges into $\xi_i$ and replacing all white vertices 
of types $\hat{r}_3^{\pm 1}$ and $\hat{r}_5^{\pm 1}$ with $\mathcal{C}_0$--subcharts 
as in Figure \ref{replaceA} and Figure \ref{replaceB}. 
Therefore $w(f,\Phi)$ is equal to 
the number modulo $2$ of white vertices of types $\hat{r}_4^{\pm 1}$ 
included in $\hat{\Gamma}$. 
\end{rem}

\begin{figure}[ht!]
\labellist
\footnotesize \hair 2pt
\pinlabel $1$ [r] at 0 119
\pinlabel $2g+1$ [r] at 0 98
\pinlabel $2g+1$ [r] at 0 88
\pinlabel $1$ [r] at 0 71
\pinlabel $1$ [r] at 0 61
\pinlabel $2g+1$ [r] at 0 42
\pinlabel $2g+1$ [r] at 0 32
\pinlabel $1$ [r] at 0 14
\pinlabel $\hat{r}_3$ [r] at 46 66
\pinlabel $1$ [r] at 190 126
\pinlabel $2g+1$ [r] at 190 105
\pinlabel $2g+1$ [r] at 190 95
\pinlabel $1$ [r] at 190 77
\pinlabel $1$ [r] at 190 57
\pinlabel $2g+1$ [r] at 190 36
\pinlabel $2g+1$ [r] at 190 26
\pinlabel $1$ [r] at 190 9
\pinlabel $r_3$ [r] at 236 101
\pinlabel $r_3$ [r] at 236 32
\endlabellist
\centering
\includegraphics[scale=0.75]{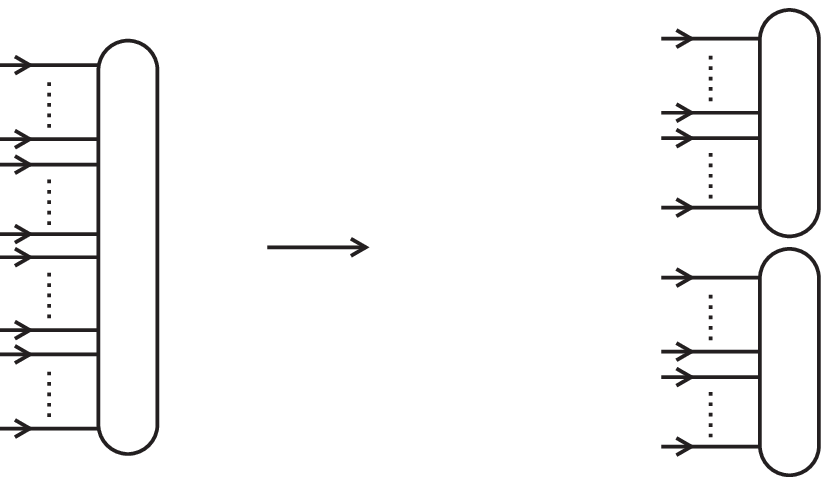}
\caption{Replacing a vertex of type $\hat{r}_3$}
\label{replaceA}
\end{figure}

\begin{figure}[ht!]
\labellist
\footnotesize \hair 2pt
\pinlabel $1$ [r] at 0 86
\pinlabel $2g+1$ [r] at 0 65
\pinlabel $2g+1$ [r] at 0 55
\pinlabel $1$ [r] at 0 37
\pinlabel $1$ [l] at 74 86
\pinlabel $2g+1$ [l] at 74 65
\pinlabel $2g+1$ [l] at 74 55
\pinlabel $1$ [l] at 74 37
\pinlabel $i$ [l] at 40 113
\pinlabel $i$ [l] at 40 7
\pinlabel $\hat{r}_5$ [r] at 45 60
\pinlabel $1$ [r] at 194 86
\pinlabel $2g+1$ [r] at 194 65
\pinlabel $2g+1$ [r] at 194 55
\pinlabel $1$ [r] at 194 37
\pinlabel $1$ [l] at 309 86
\pinlabel $2g+1$ [l] at 309 65
\pinlabel $2g+1$ [l] at 309 55
\pinlabel $1$ [l] at 309 37
\pinlabel $i$ [l] at 255 113
\pinlabel $i$ [l] at 255 7
\pinlabel $r_3$ [r] at 239 60
\pinlabel $r_3$ [r] at 278 60
\endlabellist
\centering
\includegraphics[scale=0.75]{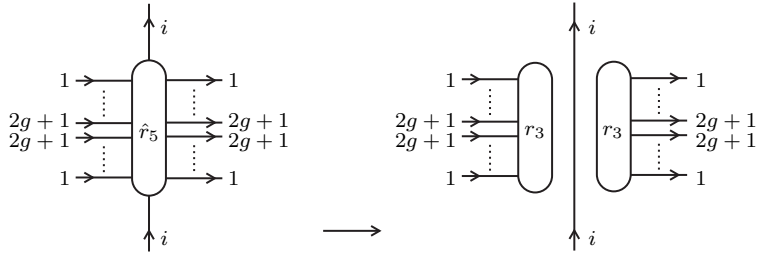}
\caption{Replacing a vertex of type $\hat{r}_5$}
\label{replaceB}
\end{figure}

Let $(f,\Phi)$ and $(f',\Phi')$ be hyperelliptic Lefschetz fibrations of genus $g$ over $B$. 

\begin{prop}\label{inv}
If $g$ is odd and $(f,\Phi)$ is $\mathcal{H}$--isomorphic to $(f',\Phi')$, 
then we have $w(f,\Phi)=w(f',\Phi')$. 
\end{prop}

\begin{proof} 
The statement follows from Propositions \ref{correspondence}, 
\ref{isomorphism}, \ref{typew}, and \ref{transition}. 
\end{proof}

Thus $w(f,\Phi)$ turns out to be an invariant of the $\mathcal{H}$--isomorphism class 
of $(f,\Phi)$. 

We next prove a stable classification theorem for hyperelliptic Lefschetz fibrations, 
which improves the stabilization theorem shown in \cite{EK2013}. 
Let $(f,\Phi)$ and $(f',\Phi')$ be hyperelliptic Lefschetz fibrations of genus $g$ over $B$. 
Let $\Gamma_0$ be a $\hat{\mathcal{C}}$--chart in $S^2$ depicted in Figure \ref{chartH} 
and $(f_0,\Phi_0)$ a hyperelliptic Lefschetz fibration described by $\Gamma_0$. 

\begin{figure}[ht!]
\labellist
\footnotesize \hair 2pt
\pinlabel $1$ [r] at 0 114
\pinlabel $2g+1$ [r] at 0 92
\pinlabel $2g+1$ [r] at 0 82
\pinlabel $1$ [r] at 0 65
\pinlabel $1$ [r] at 0 56
\pinlabel $2g+1$ [r] at 0 37
\pinlabel $2g+1$ [r] at 0 27
\pinlabel $1$ [r] at 0 8
\endlabellist
\centering
\includegraphics[scale=0.75]{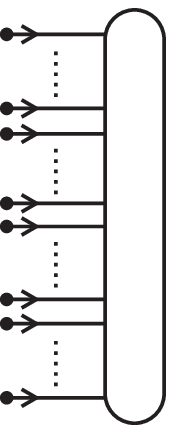}
\caption{$\hat{\mathcal{C}}$--chart $\Gamma_0$ in $S^2$}
\label{chartH}
\end{figure}

A Hurwitz system of $(f_0,\Phi_0)$ is given by 
$(\zeta_1,\zeta_2,\ldots,\zeta_{2g+1},\zeta_{2g+1},\ldots,\zeta_2,\zeta_1)^2$. 
The total space of $f_0$ is diffeomorphic to 
$\mathbb{CP}^2\# (4g+5)\overline{\mathbb{CP}}^2$, 
which is a natural generalization of the rational elliptic surface $E(1)$ 
(see Gompf and Stipsicz \cite[\S 8.4]{GS1999} and Ito \cite{Ito2002}). 

\begin{defn}\label{stiso} 
We say that $(f,\Phi)$ is {\it stably isomorphic} to $(f',\Phi')$ 
if there exists a non-negative integer $N$ such that an $\mathcal{H}$--fiber sum 
$(f\# Nf_0,\Phi)$ is $\mathcal{H}$--isomorphic 
to an $\mathcal{H}$--fiber sum $(f'\# Nf_0, \Phi')$. 
\end{defn}

\begin{rem} 
If $N$ is positive, then the $\mathcal{H}$--isomorphism class of an $\mathcal{H}$--fiber sum 
$(f\# Nf_0,\Phi)$ does not depend on choices of orientation preserving diffeomorphisms 
of $\Sigma_g$ used for the construction of $(f\# Nf_0,\Phi)$ 
(see Proof of Theorem \ref{stable}). 
\end{rem}

We give a complete classification of the stable isomorphism classes of 
hyperelliptic Lefschetz fibrations of genus $g$ over $B$. 

\begin{thm}\label{stable}
Let $(f,\Phi)$ and $(f',\Phi')$ be 
hyperelliptic Lefschetz fibrations of genus $g$ over $B$. 
Then $(f,\Phi)$ is stably isomorphic to $(f',\Phi')$ if and only if 
the following conditions hold: 
{\rm (i)} $n_0^{\pm}(f)=n_0^{\pm}(f');$ 
{\rm (ii)} $n_h^{\pm}(f)=n_h^{\pm}(f')$ for every $h=1,\ldots ,[g/2];$ 
{\rm (iii)} $w(f,\Phi)=w(f',\Phi')$ if $g$ is odd. 
\end{thm}

\begin{proof} 
We first prove the `if' part. 
Assume that $(f,\Phi)$ and $(f',\Phi')$ satisfy the conditions (i), (ii), and (iii). 
Let $\Gamma$ and $\Gamma'$ be $\hat{\mathcal{C}}$--charts in $B$ 
corresponding to $(f,\Phi)$ and $(f',\Phi')$, respectively. 
Since every edge has two adjacent vertices, 
the sum of the signed numbers of adjacent edges for all vertices of $\Gamma$ 
is equal to zero: 
\[
4(2g+1)m_3(\Gamma)+2(g+1)(2g+1)m_4(\Gamma)-\sum_{i=1}^{2g+1}n_0(i)(\Gamma)
-4\sum_{h=1}^{[g/2]}h(2h+1) n_h(\Gamma)=0. 
\]
A similar equality for $\Gamma'$ also holds. 
Interpreting the conditions (i) and (ii) as conditions on $\Gamma$ and $\Gamma'$, 
we have $\sum_{i=1}^{2g+1}n_0(i)(\Gamma)=\sum_{i=1}^{2g+1}n_0(i)(\Gamma')$ 
and $n_h(\Gamma)=n_h(\Gamma')$ for $h=1,\ldots ,[g/2]$. 
Thus we obtain 
\begin{equation}
2m_3(\Gamma)+(g+1)m_4(\Gamma)=2m_3(\Gamma')+(g+1)m_4(\Gamma'). \tag{$*$}
\end{equation}

Let $N$ be an integer larger than both of the number of edges of $\Gamma$ 
and that of $\Gamma'$. 
Choose a base point $b_0\in B-(\Gamma\cup\Gamma')$. 
The $\mathcal{H}$--fiber sum $(f\# Nf_0,\Phi)$ is described by a chart 
$(\cdots ((\Gamma\#_{w_1}\Gamma_0)\#_{w_2}\Gamma_0)\cdots )\#_{w_N}\Gamma_0$ 
for some words $w_1,\ldots ,w_N$ in $\hat{\mathcal{X}}\cup\hat{\mathcal{X}}^{-1}$. 
Since hoops surrounding $\Gamma_0$ can be removed by use of the edges of $\Gamma_0$ 
as in Figure \ref{removeA}, 
the chart is transformed into a product $\Gamma\oplus N\Gamma_0$ 
by channel changes. Similarly, 
the $\mathcal{H}$--fiber sum $f'\# Nf_0$ is described by a product 
$\Gamma'\oplus N\Gamma_0$. 

\begin{figure}[ht!]
\labellist
\footnotesize \hair 2pt
\pinlabel $i$ [b] at 41 71
\pinlabel $1$ [b] at 23 52
\pinlabel $i$ [b] at 43 52
\pinlabel $2g\negthinspace +\negthinspace 1$ [b] at 61 52
\pinlabel $T_0$ at 43 11
\pinlabel $i$ [b] at 188 72
\pinlabel $1$ [b] at 142 52
\pinlabel $2g\negthinspace +\negthinspace 1$ [b] at 179 52
\pinlabel $T_0$ at 161 11
\pinlabel $1$ [b] at 257 52
\pinlabel $i$ [b] at 275 52
\pinlabel $2g\negthinspace +\negthinspace 1$ [b] at 293 52
\pinlabel $T_0$ at 276 11
\endlabellist
\centering
\includegraphics[scale=0.75]{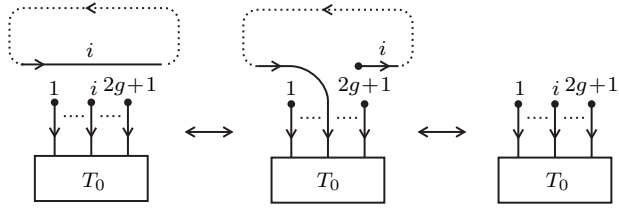}
\caption{Removing a hoop}
\label{removeA}
\end{figure}

We choose and fix $2g+1$ edges of $\Gamma_0$ which are 
labeled with $1,2,\ldots ,2g+1$ and adjacent to black vertices. 
We apply chart moves only to these edges in the following. 
Since $\Gamma_0$ can pass through any edge of $\Gamma$ as shown in Figure \ref{pass}, 
we can move $\Gamma_0$ to any region of $B-\Gamma$ by channel changes. 
\begin{figure}[ht!]
\labellist
\footnotesize \hair 2pt
\pinlabel $i$ [b] at 35 69
\pinlabel $1$ [b] at 18 51
\pinlabel $i$ [b] at 36 52
\pinlabel $2g\negthinspace +\negthinspace 1$ [b] at 55 51
\pinlabel $T_0$ at 35 11
\pinlabel {(a)} [t] at 37 0
\pinlabel $i$ [b] at 106 72
\pinlabel $i$ [b] at 160 72
\pinlabel $1$ [b] at 115 52
\pinlabel $2g\negthinspace +\negthinspace 1$ [b] at 151 52
\pinlabel $T_0$ at 134 11
\pinlabel {(b)} [t] at 134 0
\pinlabel $i$ [t] at 200 66
\pinlabel $i$ [t] at 255 66
\pinlabel $2g\negthinspace +\negthinspace 1$ [t] at 208 86
\pinlabel $1$ [t] at 246 86
\pinlabel $T_0$ at 228 126
\pinlabel {(c)} [t] at 228 0
\pinlabel $i$ [t] at 322 66
\pinlabel $2g\negthinspace +\negthinspace 1$ [t] at 303 86
\pinlabel $i$ [t] at 321 85
\pinlabel $1$ [t] at 341 86
\pinlabel $T_0$ at 322 126
\pinlabel {(d)} [t] at 321 0
%{\rotatebox{180}{$M$}}
\endlabellist
\centering
\includegraphics[scale=0.75]{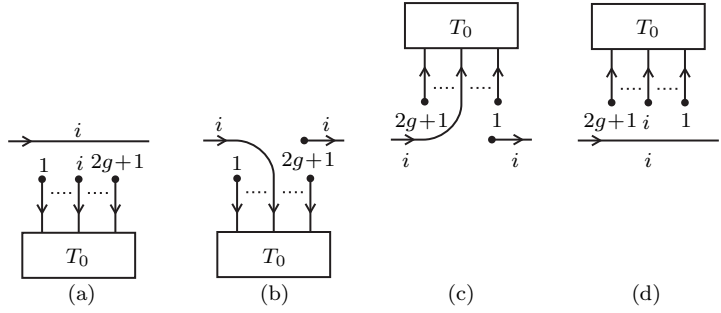}
\vspace{3mm}
\caption{Passing through an edge}
\label{pass}
\end{figure}
For each edge of $\Gamma$, we move a copy of $\Gamma_0$ to a region adjacent to 
the edge and apply a channel change to the edge and $\Gamma_0$ 
as in Figure \ref{pass} (a) and (b). 
Applying chart moves of transition to each component of the chart 
as in Figure \ref{movesE}, 
we remove white vertices of type $\hat{r}_1(i,j)^{\pm 1}, \hat{r}_2(i)^{\pm 1},\hat{r}_5^{\pm 1}$ 
to obtain a union of copies of 
$L_0(i), L_h, L'_h, R_3, R'_3, R_4, R'_4, \Gamma_0$ shown in Figures 
\ref{chartI}, \ref{chartJ}, 
where we use a simplification of diagrams as in Figure \ref{symbol}. 
Although $R_3$ is the same chart as $\Gamma_0$, 
we distinguish $R_3$ from $\Gamma_0$s 
which do not come from vertices of type $r_3$ in $\Gamma$. 

\begin{figure}[ht!]
\labellist
\footnotesize \hair 2pt
\pinlabel $j$ [br] at 2 153
\pinlabel $i$ [tr] at 0 123
\pinlabel $i$ [bl] at 33 153
\pinlabel $j$ [tl] at 33 125
\pinlabel $j$ [br] at 72 153
\pinlabel $i$ [bl] at 103 153
\pinlabel $j$ [tl] at 103 125
\pinlabel $j$ [br] at 140 156
\pinlabel $i$ [r] at 133 138
\pinlabel $j$ [tr] at 140 118
\pinlabel $i$ [bl] at 177 156
\pinlabel $j$ [l] at 184 138
\pinlabel $i$ [tl] at 177 120
\pinlabel $j$ [br] at 242 156
\pinlabel $i$ [r] at 235 138
\pinlabel $i$ [bl] at 278 156
\pinlabel $j$ [l] at 286 138
\pinlabel $i$ [tl] at 279 120
\pinlabel $1$ [r] at 30 77
\pinlabel $2g+1$ [r] at 30 56
\pinlabel $2g+1$ [r] at 30 47
\pinlabel $1$ [r] at 30 28
\pinlabel $i$ [l] at 72 96
\pinlabel $1$ [l] at 105 77
\pinlabel $2g+1$ [l] at 105 56
\pinlabel $2g+1$ [l] at 105 47
\pinlabel $1$ [l] at 105 28
\pinlabel $i$ [l] at 71 9
\pinlabel $1$ [r] at 175 77
\pinlabel $2g+1$ [r] at 175 56
\pinlabel $2g+1$ [r] at 175 47
\pinlabel $1$ [r] at 175 28
\pinlabel $i$ [l] at 217 93
\pinlabel $1$ [l] at 251 77
\pinlabel $2g+1$ [l] at 251 56
\pinlabel $2g+1$ [l] at 251 47
\pinlabel $1$ [l] at 251 28
\endlabellist
\centering
\includegraphics[scale=0.75]{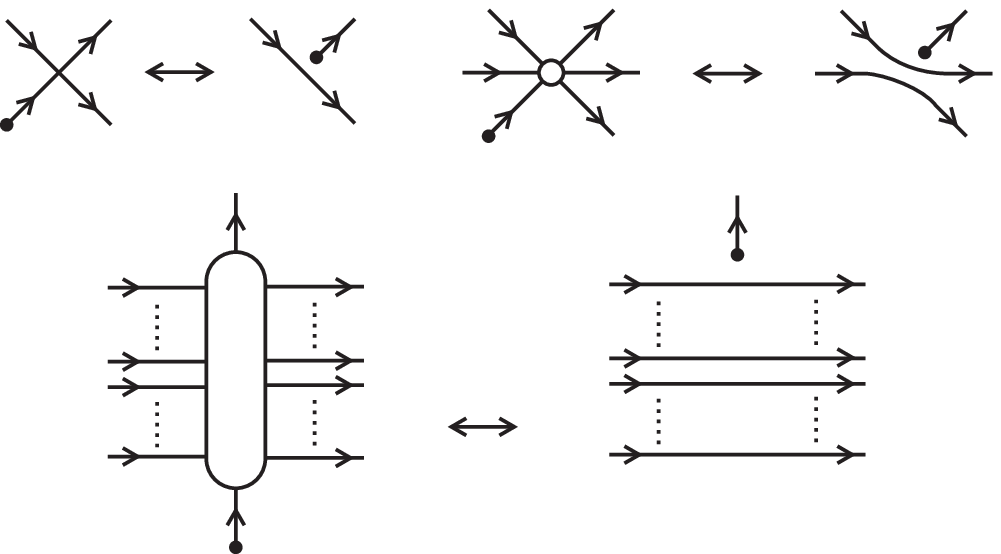}
\caption{Chart moves of transition}
\label{movesE}
\end{figure}

\begin{figure}[ht!]
\labellist
\footnotesize \hair 2pt
\pinlabel $i$ [b] at 24 64
\pinlabel $1$ [r] at 173 113
\pinlabel $2h$ [r] at 173 92
\pinlabel $1$ [r] at 173 82
\pinlabel $2h$ [r] at 173 65
\pinlabel $1$ [r] at 173 30
\pinlabel $2h$ [r] at 173 8
\pinlabel $1$ [l] at 304 113
\pinlabel $2h$ [l] at 304 92
\pinlabel $1$ [l] at 304 82
\pinlabel $2h$ [l] at 304 65
\pinlabel $1$ [l] at 304 30
\pinlabel $2h$ [l] at 304 8
\endlabellist
\centering
\includegraphics[scale=0.75]{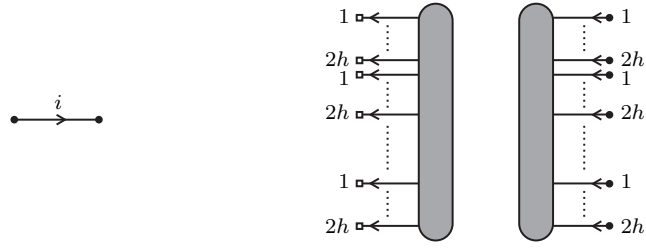}
\caption{Charts $L_0(i)$, $L_h$, $L'_h$}
\label{chartI}
\end{figure}

\begin{figure}[ht!]
\labellist
\footnotesize \hair 2pt
\pinlabel $1$ [r] at 0 113
\pinlabel $2g+1$ [r] at 0 92
\pinlabel $2g+1$ [r] at 0 82
\pinlabel $1$ [r] at 0 65
\pinlabel $1$ [r] at 0 55
\pinlabel $2g+1$ [r] at 0 37
\pinlabel $2g+1$ [r] at 0 27
\pinlabel $1$ [r] at 0 7
\pinlabel $1$ [l] at 126 113
\pinlabel $2g+1$ [l] at 126 92
\pinlabel $2g+1$ [l] at 126 82
\pinlabel $1$ [l] at 126 65
\pinlabel $1$ [l] at 126 55
\pinlabel $2g+1$ [l] at 126 37
\pinlabel $2g+1$ [l] at 126 27
\pinlabel $1$ [l] at 126 7
\pinlabel $2g+1$ [r] at 203 113
\pinlabel $1$ [r] at 203 90
\pinlabel $2g+1$ [r] at 203 56
\pinlabel $1$ [r] at 203 37
\pinlabel $2g+1$ [r] at 203 27
\pinlabel $1$ [r] at 203 8
\pinlabel $2g+1$ [l] at 327 113
\pinlabel $1$ [l] at 327 90
\pinlabel $2g+1$ [l] at 327 56
\pinlabel $1$ [l] at 327 37
\pinlabel $2g+1$ [l] at 327 27
\pinlabel $1$ [l] at 327 8
\endlabellist
\centering
\includegraphics[scale=0.75]{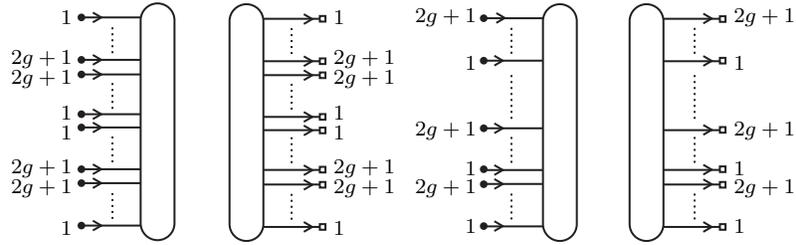}
\caption{Charts $R_3$, $R'_3$, $R_4$, $R'_4$}
\label{chartJ}
\end{figure}

\begin{figure}[ht!]
\labellist
\footnotesize \hair 2pt
\pinlabel $i$ [b] at 3 45
\pinlabel $0$ [b] at 89 57
\pinlabel $i$ [b] at 107 71
\pinlabel $2g$ [b] at 126 57
\pinlabel $T_0$ at 108 14
\pinlabel $=$ at 40 28
\endlabellist
\centering
\includegraphics[scale=0.75]{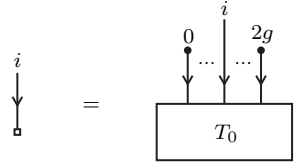}
\caption{Simplification of diagram}
\label{symbol}
\end{figure}

If there is a pair of $R_3$ and $R'_3$, 
we remove them by a death of a pair of white vertices 
to obtain $4(2g+1)$ copies of $\Gamma_0$. 
Similarly, we remove a pair of $R_4$ and $R'_4$ 
to obtain $2(g+1)(2g+1)$ copies of $\Gamma_0$. 
Since there is at least one $\Gamma_0$, any copy of $L_0(i)$ can be transformed into 
$L_0(1)$ as in Figure \ref{label}. 

\begin{figure}[ht!]
\labellist
\footnotesize \hair 2pt
\pinlabel $j$ [l] at 28 152
\pinlabel $0$ [b] at 83 186
\pinlabel $i$ [b] at 84 201
\pinlabel $2g$ [b] at 120 186
\pinlabel $T_0$ at 101 143
\pinlabel $j$ [l] at 199 152
\pinlabel $i$ [l] at 212 164
\pinlabel $0$ [b] at 255 186
\pinlabel $j$ [b] at 255 201
\pinlabel $2g$ [b] at 292 186
\pinlabel $T_0$ at 274 143
\pinlabel $j$ [l] at 370 152
\pinlabel $i$ [l] at 382 150
\pinlabel $j$ [l] at 390 142
\pinlabel $0$ [b] at 426 186
\pinlabel $j$ [b] at 444 186
\pinlabel $2g$ [b] at 463 186
\pinlabel $T_0$ at 444 143
\pinlabel $i$ [b] at 26 91
\pinlabel $j$ [t] at 26 46
\pinlabel $i$ [b] at 26 23
\pinlabel $j$ [t] at 26 13
\pinlabel $0$ [b] at 82 73
\pinlabel $j$ [b] at 100 73
\pinlabel $2g$ [b] at 119 73
\pinlabel $T_0$ at 100 29
\pinlabel $i$ [l] at 202 82
\pinlabel $j$ [tl] at 216 8
\pinlabel $i$ [b] at 201 10
\pinlabel $0$ [b] at 254 73
\pinlabel $j$ [b] at 273 73
\pinlabel $2g$ [b] at 291 73
\pinlabel $T_0$ at 273 29
\pinlabel $i$ [l] at 372 44
\pinlabel $0$ [b] at 425 73
\pinlabel $j$ [b] at 444 73
\pinlabel $2g$ [b] at 463 73
\pinlabel $T_0$ at 443 29
\endlabellist
\centering
\includegraphics[scale=0.7]{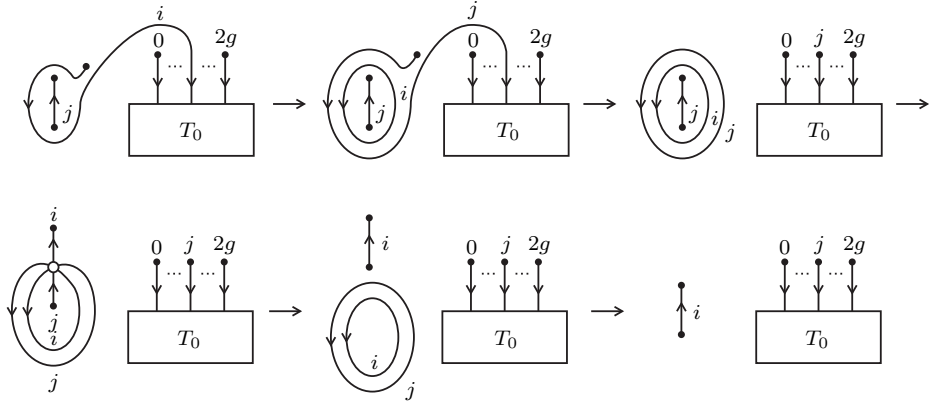}
\caption{Changing a label ($j=i+1$)}
\label{label}
\end{figure}

Thus we have a union $\Gamma_1$ of 
$n_0^-(\Gamma)$ copies of $L_0(1)$, 
$n_h^+(\Gamma)$ copies of $L_h$, 
$n_h^-(\Gamma)$ copies of $L'_h$, 
$|m_3(\Gamma)|$ copies of $R_3$ (or $R'_3$), 
$|m_4(\Gamma)|$ copies of $R_4$ (or $R'_4$), 
and $k$ copies of $\Gamma_0$ for some $k$. 
A similar argument implies that $\Gamma'\oplus N\Gamma_0$ is transformed into 
a union $\Gamma'_1$ of 
$n_0^-(\Gamma')$ copies of $L_0(1)$, 
$n_h^+(\Gamma')$ copies of $L_h$, 
$n_h^-(\Gamma')$ copies of $L'_h$, 
$|m_3(\Gamma')|$ copies of $R_3$ (or $R'_3$), 
$|m_4(\Gamma')|$ copies of $R_4$ (or $R'_4$), 
and $k'$ copies of $\Gamma_0$ for some $k'$ 
by chart moves of type W and chart moves of transition. 
By virtue of the conditions (i) and (ii) together with 
$n_0^+(\Gamma\oplus N\Gamma_0)=n_0^+(\Gamma_1)$, 
$n_0^+(\Gamma'\oplus N\Gamma_0)=n_0^+(\Gamma'_1)$, 
and the equality $(*)$, 
we conclude that $k=k'$ because of $n_0^+(\Gamma_0)\ne 0$. 
Hence $\Gamma_1$ and $\Gamma'_1$ have the same numbers of copies of 
$L_0(1)$, $L_h$, $L'_h$, and $\Gamma_0$, whereas they may have different numbers of 
copies of $R_3$ (or $R'_3$) and $R_4$ (or $R'_4$). 

We can suppose that $m_4(\Gamma)\geq m_4(\Gamma')$ without loss of generality. 
The number $m_4(\Gamma)-m_4(\Gamma')$ is even 
because of the equality $(*)$ for even $g$ and because of the condition (iii) for odd $g$. 
We put $2\ell:=m_4(\Gamma)-m_4(\Gamma')$. 
Taking $N$ large enough, we can assume that $k>4\ell (g+1)(2g+1)$. 
Applying births of $2\ell$ pairs of white vertices of type $\hat{r}_4^{\pm 1}$ 
to $4\ell (g+1)(2g+1)$ copies of $\Gamma_0$ included in $\Gamma_1$, 
we obtain $2\ell$ pairs of $R_4$ and $R'_4$. 
Since $2R_4$ is transformed into $(g+1)R_3$ by a sequence of chart moves of type W 
and chart moves of transition (see \cite[Lemma 4.1]{Endo2008}), 
we obtain $(g+1)\ell R_3$ from $2\ell R_4$. 
We thus have a new chart $\Gamma_2$. 
Removing pairs of $R_3$ and $R'_3$ and those of $R_4$ and $R'_4$ 
included in $\Gamma_2$ if necessary, 
$\Gamma_2$ and $\Gamma'_1$ have the same numbers of copies of 
$L_0(1)$, $L_h$, $L'_h$, $R_3$ (or $R'_3$), $R_4$ (or $R'_4$), and $\Gamma_0$ 
because of the equality $(*)$. 
Then $\Gamma_2$ is transformed into $\Gamma'_1$ by an ambient isotopy of $B$ 
relative to $b_0$, 
which means that $\Gamma\oplus N\Gamma_0$ is transformed into 
$\Gamma'\oplus N\Gamma_0$ by chart moves of type W, chart moves of transition, 
and ambient isotopies of $B$ relative to $b_0$. 
Therefore $f\# Nf_0$ is (strictly) $\mathcal{H}$--isomorphic to $f'\# Nf_0$ by Proposition 
\ref{isomorphism}. 

We next prove the `only if' part. 
Take a non-negative integer $N$ so that 
$(f\# Nf_0,\Phi)$ is $\mathcal{H}$--isomorphic to $(f'\# Nf_0,\Phi')$. 
Since an $\mathcal{H}$--isomorphism preserves numbers 
and types of vanishing cycles and the invariant $w$ (see Proposition \ref{inv}), 
we have $n_0^{\pm}(f\# Nf_0)=n_0^{\pm}(f'\# Nf_0)$, 
$n_h^{\pm}(f\# Nf_0)=n_h^{\pm}(f'\# Nf_0)$ for every $h=1,\ldots ,[g/2]$, 
and $w(f\# Nf_0,\Phi)=w(f'\# Nf_0,\Phi')$. 
The conditions (i), (ii), (iii) follows from additivity of $n_0^{\pm}, n_h^{\pm}, w$ 
under $\mathcal{H}$--fiber sum. 
\end{proof}

\begin{rem} 
Theorem \ref{stable} is also valid for any hyperelliptic Lefschetz fibration $(f_0,\Phi_0)$ 
which satisfies the following conditions: 
(i) $f_0$ is chiral and irreducible; 
(ii) the base space of $f_0$ is $S^2$; 
(iii) a Hurwitz system of $f_0$ with respect to $\Phi_0$ includes at least one 
$\zeta_i$ for every $i=1,\ldots ,2g+1$ 
(cf. Auroux \cite{Auroux2005} and \cite[Definition 4.1]{EHKT2014}). 
\end{rem}

\begin{rem}\label{as}
Auroux and Smith \cite[\S 2.1]{AS2008} remarked that 
the existence of a certain $\mathbb{Z}_2$-valued 
invariant and a stabilization theorem for hyperelliptic Lefschetz fibrations 
follow from a result of Kharlamov and Kulikov \cite{KK2003} about 
braid monodromy factorizations. 
Although they did not give any precise definition of the invariant, 
we expect that their invariant would be the same as the invariant $w$ and their theorem 
would be similar to (a part of) Theorem \ref{stable}. 
\end{rem}

%%%%%%%%%%%%% Section 6 %%%%%%%%%%%%%%%

\section{Examples and remarks}

%%%%%%%%%%%%%%%%%%%%%%%%%%%%%%%%%

In this section we exhibit examples of stabilizations of hyperelliptic Lefschetz fibrations 
and examples of pairs of hyperelliptic Lefschetz fibrations which are not stably equivalent, 
and make a few remarks. 
We do not specify hyperelliptic structures of Lefschetz fibrations in the following 
because they are obvious from given Hurwitz systems. 

We first show an example of non-isomorphic pair of irreducible chiral 
hyperelliptic Lefschetz fibrations which become isomorphic after one stabilization. 

\begin{exmp}
Let $g$ be an integer greater than one. 
We consider the following Hurwitz systems. 
\[
C_{\rm II} := (\zeta_1, \zeta_2, \ldots , \zeta_{2g})^{4g+2}, \quad 
I^{2g} := (\zeta_1,\zeta_2,\ldots,\zeta_{2g+1},\zeta_{2g+1},\ldots,\zeta_2,\zeta_1)^{2g}. 
\]
Let $f_{C_{\rm II}}:M_{C_{\rm II}}\rightarrow S^2$ and 
$f_{I^{2g}}:M_{I^{2g}}\rightarrow S^2$ be the hyperelliptic Lefschetz fibrations of genus $g$ 
determined by $C_{\rm II}$ and $I^{2g}$, respectively. 
$f_{I^{2g}}$ is nothing but the (untwisted) $\mathcal{H}$--fiber sum of $g$ copies 
of $f_0$ in Section \ref{section5}. 
Both $f_{C_{\rm II}}$ and $f_{I^{2g}}$ are irreducible, chiral 
and have $4g(2g+1)$ singular fibers. 
They are not $\mathcal{H}$--isomorphic 
because the images of their monodromy 
representations are different (see \cite[Example 4.14]{Endo2000}). 
Moreover the manifolds $M_{C_{\rm II}}$ and $M_{I^{2g}}$ are 
homeomorphic but not diffeomorphic by Freedman's theorem and a theorem of 
Usher \cite{Usher2006} (see \cite[Remark 4.9]{Endo2008}). 
In particular $f_{C_{\rm II}}$ and $f_{I^{2g}}$ are not isomorphic. 
On the other hand, the (untwisted) $\mathcal{H}$--fiber sums 
$f_{C_{\rm II}}\# f_0$ and $f_{I^{2g}}\# f_0$ are $\mathcal{H}$--isomorphic 
because a Hurwitz system of $f_{C_{\rm II}}\# f_0$ is transformed into 
that of $f_{I^{2g}}\# f_0$ by elementary transformations 
and simultaneous conjugations (see \cite[Lemma 4.1]{Endo2008}). 
\end{exmp}

We next show an example of non-isomorphic pair of chiral 
hyperelliptic Lefschetz fibrations with singular fiber of type ${\rm II}_{g/2}$ 
which become isomorphic after one stabilization. 

\begin{exmp}
Let $g$ be an even integer greater than one. 
We consider the following Hurwitz systems. 
%{\allowdisplaybreaks %
\begin{align*}
PJ & := ((\zeta_1\zeta_2\cdots\zeta_{2g})^{2g+2}, 
(\zeta'_{g+1},\ldots ,\zeta'_3,\zeta'_2), \ldots ,(\zeta'_{2g},\ldots ,\zeta'_{g+2},\zeta'_{g+1}), \\
& \quad (\zeta_{g+1},\ldots ,\zeta_3,\zeta_2), \ldots ,(\zeta_{2g},\ldots ,\zeta_{g+2},\zeta_{g+1}), 
(\zeta_1,\zeta_2,\ldots ,\zeta_{2g})^{2g+1}), \\
{\rm where}\; & \; \zeta'_i:=(\zeta_g^{-1}\cdots\zeta_2^{-1}\zeta_1^{-1})^{g+1}\zeta_i
(\zeta_1\zeta_2\cdots \zeta_g)^{g+1}, \; {\rm and} \\
RI^{g-1} & := ((\zeta_g\cdots\zeta_2\zeta_1)^{2g+2}, 
\zeta_1^{-1}\zeta_2^{-1}\cdots\zeta_g^{-1}\zeta_{g+1}\zeta_g\cdots\zeta_2\zeta_1, \\
& \quad \zeta_2^{-1}\zeta_3^{-1}\cdots\zeta_{g+1}^{-1}\zeta_{g+2}
\zeta_{g+1}\cdots\zeta_3\zeta_2, \ldots ,
\zeta_{g+1}^{-1}\zeta_{g+2}^{-1}\cdots\zeta_{2g}^{-1}\zeta_{2g+1}
\zeta_{2g}\cdots\zeta_{g+2}\zeta_{g+1}, \\
& \quad 
(\zeta_1,\zeta_2,\ldots,\zeta_{2g+1},\zeta_{2g+1},\ldots,\zeta_2,\zeta_1)^{g-1} ). 
\end{align*}
%}
Let $f_{PJ}:M_{PJ}\rightarrow S^2$ and 
$f_{RI^{g-1}}:M_{RI^{g-1}}\rightarrow S^2$ be the hyperelliptic Lefschetz fibrations of genus $g$ 
determined by $PJ$ and $RI^{g-1}$, respectively. 
Both $PJ$ and $RI^{g-1}$ are chiral and have $6g^2+2g+1$ singular fibers. 
One singular fiber is of type ${\rm II}_{g/2}$ and the others are of type ${\rm I}$. 
The manifolds $M_{PJ}$ and $M_{RI^{g-1}}$ are 
homeomorphic but not diffeomorphic (see \cite[Theorem 4.8]{Endo2008} for $g\geq 4$ 
and Sato \cite[Answer to Question 5.1]{Sato2013} for $g=2$). 
In particular $f_{PJ}$ and $f_{RI^{g-1}}$ are not isomorphic. 
On the other hand, the (untwisted) $\mathcal{H}$--fiber sums 
$f_{PJ}\# f_0$ and $f_{RI^{g-1}}\# f_0$ are $\mathcal{H}$--isomorphic 
because a Hurwitz system of $f_{PJ}\# f_0$ 
is transformed into that of $f_{RI^{g-1}}\# f_0$ by elementary transformations 
and simultaneous conjugations (see \cite[Theorem 4.10]{Endo2008}). 
\end{exmp}

We then show an example of pair of irreducible chiral 
hyperelliptic Lefschetz fibrations with the same number of singular fibers 
which are not stably isomorphic. 

\begin{exmp}\label{not1}
Let $g$ be an odd integer greater than one. 
We consider the following Hurwitz systems. 
\[
C_{\rm I} := (\zeta_1, \zeta_2, \ldots , \zeta_{2g+1})^{2g+2}, \quad 
I^{g+1} := (\zeta_1,\zeta_2,\ldots,\zeta_{2g+1},\zeta_{2g+1},\ldots,\zeta_2,\zeta_1)^{g+1}. 
\]
Let $f_{C_{\rm I}}:M_{C_{\rm I}}\rightarrow S^2$ and 
$f_{I^{g+1}}:M_{I^{g+1}}\rightarrow S^2$ be the hyperelliptic Lefschetz fibrations of genus $g$ 
determined by $C_{\rm I}$ and $I^{g+1}$, respectively. 
$f_{I^{g+1}}$ is nothing but the (untwisted) $\mathcal{H}$--fiber sum of $(g+1)/2$ copies 
of $f_0$ in Section \ref{section5}. 
Both $f_{C_{\rm I}}$ and $f_{I^{g+1}}$ are irreducible, chiral and have $2(g+1)(2g+1)$ 
singular fibers. 
$\hat{\mathcal{C}}$--charts corresponding to $f_{C_{\rm I}}$ and $f_{I^{g+1}}$ are 
$R_4$ and $((g+1)/2)R_3$ ($=((g+1)/2)\Gamma_0$), respectively 
(see Figure \ref{chartJ} and Figure \ref{chartH}). 
The values of the invariant $w$ for $f_{C_{\rm I}}$ and $f_{I^{g+1}}$ are computed as follows 
(see Definition \ref{w} and Definition \ref{w2}). 
\[
w(f_{C_{\rm I}})=w(R_4)=1, \quad 
w(f_{I^{g+1}})=w\left( \left(\frac{g+1}{2}\right)R_3\right)=\frac{g+1}{2}w(R_3)=0. 
\]
Therefore $f_{C_{\rm I}}$ and $f_{I^{g+1}}$ are not stably isomorphic 
by Theorem \ref{stable}. Since the images of monodromy representations of 
the (untwisted) $\mathcal{H}$--fiber sums $f_{C_{\rm I}}\# Nf_0$ and $f_{I^{g+1}}\# Nf_0$ 
coincide with $\mathcal{H}_g$ for any positive integer $N$, 
they are not even isomorphic by virtue of Proposition \ref{image}. 

Both manifolds $M_{C_{\rm I}}$ and $M_{I^{g+1}}$ are simply-connected and 
have the Euler characteristic $2(2g^2+g+3)$ and signature $-2(g+1)^2$. 
If $g\equiv 3\, ({\rm mod}\, 4)$, then $M_{I^{g+1}}$ is spin while $M_{C_{\rm I}}$ is not. 
Hence they are not homeomorphic. 
If $g\equiv 1\, ({\rm mod}\, 4)$, then both $M_{C_{\rm I}}$ and $M_{I^{g+1}}$ are not spin, 
and they are homeomorphic by Freedman's theorem. 
Since $M_{C_{\rm I}}$ has a $(-1)$--section and $M_{I^{g+1}}$ is a non-trivial fiber sum, 
they are not diffeomorphic by a theorem of Usher \cite{Usher2006}. 
The fact that $f_{C_{\rm I}}$ and $f_{I^{g+1}}$ are not isomorphic also follows from this 
observation, or from a result of Smith \cite{Smith2001'} and Stipsicz \cite{Stipsicz2001'}. 
\end{exmp}

We lastly show an example of pair of chiral 
hyperelliptic Lefschetz fibrations with the same numbers of singular fibers of each type 
which have a singular fiber of type ${\rm II}_{(g-1)/2}$ and are not stably isomorphic. 

\begin{exmp}\label{not2}
Let $g$ be an odd integer greater than one. 
We consider the following Hurwitz systems. 
%{\allowdisplaybreaks %
\begin{align*}
Q & := ((\zeta_1,\zeta_2,\ldots ,\zeta_{2g+1})^{g+2}, 
(\zeta_{g-1}\cdots\zeta_2\zeta_1)^{2g}, (\zeta_{g-1},\ldots ,\zeta_2,\zeta_1)^2, \\
& \quad \zeta_1^{-1}\zeta_2^{-1}\cdots\zeta_{g-1}^{-1}\zeta_g\zeta_{g-1}\cdots\zeta_2\zeta_1, 
\zeta_2^{-1}\zeta_3^{-1}\cdots\zeta_g^{-1}\zeta_{g+1}\zeta_g\cdots \zeta_3\zeta_2, \\
& \quad \ldots ,\zeta_{g+2}^{-1}\zeta_{g+3}^{-1}\cdots \zeta_{2g}^{-1}\zeta_{2g+1}
\zeta_{2g}\cdots \zeta_{g+3}\zeta_{g+2}), \\
R & := (\zeta_1,\zeta_2,\ldots ,\zeta_{2g+1}, 
(\zeta_{g-1}\cdots\zeta_2\zeta_1)^{2g}, (\zeta_{g-1},\ldots ,\zeta_2,\zeta_1)^2, \\
& \quad \zeta_1^{-1}\zeta_2^{-1}\cdots\zeta_{g-1}^{-1}\zeta_g\zeta_{g-1}\cdots\zeta_2\zeta_1, 
\zeta_2^{-1}\zeta_3^{-1}\cdots\zeta_g^{-1}\zeta_{g+1}\zeta_g\cdots \zeta_3\zeta_2, \\
& \quad \ldots ,\zeta_{g+2}^{-1}\zeta_{g+3}^{-1}\cdots \zeta_{2g}^{-1}\zeta_{2g+1}
\zeta_{2g}\cdots \zeta_{g+3}\zeta_{g+2}, 
(\zeta_{2g+1},\ldots ,\zeta_2,\zeta_1)^{g+1}). 
\end{align*}
%}
Let $f_Q:M_Q\rightarrow S^2$ and 
$f_R:M_R\rightarrow S^2$ be the hyperelliptic Lefschetz fibrations of genus $g$ 
determined by $Q$ and $R$, respectively. 
Both $f_Q$ and $f_R$ are chiral and have $2g^2+8g+3$ singular fibers. 
One singular fiber is of type ${\rm II}_{(g-1)/2}$ and the others are of type ${\rm I}$. 
From the constructions of $Q$ and $R$ shown in \cite[Lemma 3.4, Theorem 3.5]{Endo2008} 
together with correspondence between moves for Hurwitz systems and those for 
$\hat{\mathcal{C}}$--charts (see \cite[Lemma 20]{EK2013}), 
we can see that a $\hat{\mathcal{C}}$--chart corresponding to 
$f_Q$ has one white vertex of type $\hat{r}_4$ while 
that for $f_R$ has no white vertices of type $\hat{r}_4^{\pm 1}$. 
Thus we conclude that $w(f_Q)=1$ and $w(f_R)=0$, 
and $f_Q$ and $f_R$ are not stably isomorphic 
by Theorem \ref{stable}. Since the images of monodromy representations of the (untwisted) 
$\mathcal{H}$--fiber sums $f_Q\# Nf_0$ and $f_R\# Nf_0$ coincide with $\mathcal{H}_g$ 
for any positive integer $N$, 
they are not even isomorphic by virtue of Proposition \ref{image}. 

Both manifolds $M_Q$ and $M_R$ are simply-connected, non-spin and 
have the Euler characteristic $2g^2+4g+7$ and signature $-(g+1)^2$. 
Hence they are homeomorphic by Freedman's theorem. 
It does not seem to be known whether they are diffeomorphic. 
\end{exmp}

We end this section by making two remarks. 

\begin{rem}
For a pair of stably isomorphic hyperelliptic Lefschetz fibrations, 
it is not easy to determine the minimal number of copies of $f_0$ which 
make them $\mathcal{H}$--isomorphic. 
In particular, it does not seem to be known whether there exists a pair of 
stably isomorphic hyperelliptic Lefschetz fibrations 
which do not become $\mathcal{H}$--isomorphic after taking an $\mathcal{H}$--fiber sum 
with one copy of $f_0$. 
On the other hand, the following observation shows that 
there are many examples of non-$\mathcal{H}$-isomorphic 
hyperelliptic Lefschetz fibrations with the same base, the same fiber, 
and the same numbers of singular fibers of each type 
which become $\mathcal{H}$-isomorphic after one stabilization. 

Let $g$ be an integer greater than one and 
$B_1,\ldots ,B_r$ connected closed oriented surfaces. 
We consider a hyperelliptic Lefschetz fibration $f_i\co M_i\rightarrow B_i$ of genus $g$ 
for each $i\in\{1,\ldots ,r\}$, 
and (possibly different) $\mathcal{H}$-fiber sums $f$ and $f'$ of $f_1,\ldots ,f_r$. 
It is shown by the same argument as the proof of \cite[Proposition 4.10]{EHKT2014} 
that $\mathcal{H}$-fiber sums $f\# f_0$ and $f'\# f_0$ are isomorphic to each other. 
For example, various $\mathcal{H}$-fiber sums of (generalizations of) Matsumoto's fibration 
studied by Ozbagci and Stipsicz \cite{OS2000} and Korkmaz \cite{Korkmaz2001} 
become $\mathcal{H}$-isomorphic after one stabilization. 
\end{rem}

\begin{rem}
We consider the following Hurwitz systems in $\mathcal{H}_3$. 
%{\allowdisplaybreaks %
\begin{align*}
U & := ((\zeta_1\zeta_2)^6, 
\zeta'_3,\zeta'_2,\zeta'_1,\zeta'_4,\zeta'_3,\zeta'_2,\zeta'_5,\zeta'_4,\zeta'_3,
\zeta'_6,\zeta'_5,\zeta'_4,\zeta'_7,\zeta'_6,\zeta'_5, \\
& \quad 
\zeta_3,\zeta_2,\zeta_1,\zeta_4,\zeta_3,\zeta_2,\zeta_5,\zeta_4,\zeta_3,
\zeta_6,\zeta_5,\zeta_4,\zeta_7,\zeta_6,\zeta_5,
\zeta_1,\ldots ,\zeta_7,\zeta_1,\ldots ,\zeta_7), \\
V & := ((\zeta_1\zeta_2)^6, 
\zeta'_3,\zeta'_2,\zeta'_1,\zeta'_4,\zeta'_3,\zeta'_2,\zeta'_5,\zeta'_4,\zeta'_3,
\zeta'_6,\zeta'_5,\zeta'_4,\zeta'_7,\zeta'_6,\zeta'_5, \\
& \quad 
\zeta_3,\zeta_2,\zeta_1,\zeta_4,\zeta_3,\zeta_2,\zeta_5,\zeta_4,\zeta_3,
\zeta_6,\zeta_5,\zeta_4,\zeta_7,\zeta_6,\zeta_5,
\zeta''_1,\ldots ,\zeta''_7, \zeta_1,\ldots ,\zeta_7), 
\end{align*}
%} 
where $\zeta'_i:=(\zeta_2^{-1}\zeta_1^{-1})^3\zeta_i(\zeta_1\zeta_2)^3$ and 
$\zeta''_i:=(\zeta_7^{-1}\cdots\zeta_1^{-1})^3\zeta_i(\zeta_1\cdots\zeta_7)^3$. 
Let $f_U:M_U\rightarrow S^2$ and 
$f_V:M_V\rightarrow S^2$ be the hyperelliptic Lefschetz fibrations of genus $3$ 
determined by $U$ and $V$, respectively. 
Both $f_U$ and $f_V$ are chiral and have $45$ singular fibers. 
One singular fiber is of type ${\rm II}_1$ and the others are of type ${\rm I}$. 
Both manifolds $M_U$ and $M_V$ are simply-connected, non-spin and 
have the Euler characteristic $37$ and signature $-25$. 
Hence they are homeomorphic by Freedman's theorem. 
$\hat{\mathcal{C}}$--charts corresponding to $f_U$ and $f_V$ are 
$\Gamma_U$ and $\Gamma_V$ depicted in Figure \ref{chartM} and Figure \ref{chartN}, 
respectively (see also Figure \ref{chartL}). 

\begin{figure}[ht!]
\labellist
\footnotesize \hair 2pt
\pinlabel $1$ [r] at 9 120
\pinlabel $2$ [r] at 19 120
\pinlabel $1$ [r] at 27 120
\pinlabel $2$ [r] at 36 120
\pinlabel $1$ [r] at 45 120
\pinlabel $2$ [r] at 53 120
\pinlabel $3$ [b] at 60 105
\pinlabel $2$ [b] at 69 105
\pinlabel $1$ [b] at 78 105
\pinlabel $7$ [b] at 111 105
\pinlabel $6$ [b] at 120 105
\pinlabel $5$ [b] at 129 105
\pinlabel $1$ [r] at 154 120
\pinlabel $2$ [r] at 163 120
\pinlabel $1$ [r] at 172 120
\pinlabel $2$ [r] at 180 120
\pinlabel $1$ [r] at 189 120
\pinlabel $2$ [r] at 197 120
\pinlabel $3$ [b] at 205 105
\pinlabel $2$ [b] at 213 105
\pinlabel $1$ [b] at 221 105
\pinlabel $7$ [b] at 256 105
\pinlabel $6$ [b] at 265 105
\pinlabel $5$ [b] at 273 105
\pinlabel $\Omega$ [b] at 70 53
\pinlabel $\Omega$ [b] at 212 53
\pinlabel $1$ [r] at 9 23
\pinlabel $7$ [r] at 43 23
\pinlabel $1$ [r] at 53 23
\pinlabel $7$ [r] at 86 23
\pinlabel $1$ [r] at 95 23
\pinlabel $7$ [r] at 128 23
\pinlabel $1$ [r] at 154 23
\pinlabel $7$ [r] at 188 23
\pinlabel $1$ [r] at 197 23
\pinlabel $7$ [r] at 230 23
\pinlabel $1$ [r] at 240 23
\pinlabel $7$ [r] at 273 23
\pinlabel $1$ [b] at 290 45
\pinlabel $7$ [b] at 323 45
\pinlabel $1$ [b] at 332 45
\pinlabel $7$ [b] at 366 45
\endlabellist
\centering
\includegraphics[scale=0.75]{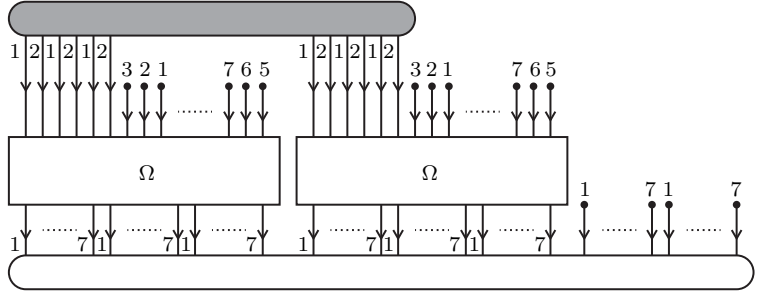}
%\bigskip
\caption{Chart $\Gamma_U$}
\label{chartM}
\end{figure}

\begin{figure}[ht!]
\labellist
\footnotesize \hair 2pt
\pinlabel $1$ [r] at 9 120
\pinlabel $2$ [r] at 19 120
\pinlabel $1$ [r] at 28 120
\pinlabel $2$ [r] at 36 120
\pinlabel $1$ [r] at 45 120
\pinlabel $2$ [r] at 53 120
\pinlabel $3$ [b] at 60 105
\pinlabel $2$ [b] at 69 105
\pinlabel $1$ [b] at 78 105
\pinlabel $7$ [b] at 111 105
\pinlabel $6$ [b] at 120 105
\pinlabel $5$ [b] at 129 105
\pinlabel $1$ [r] at 197 120
\pinlabel $2$ [r] at 206 120
\pinlabel $1$ [r] at 214 120
\pinlabel $2$ [r] at 223 120
\pinlabel $1$ [r] at 232 120
\pinlabel $2$ [r] at 240 120
\pinlabel $3$ [b] at 248 105
\pinlabel $2$ [b] at 256 105
\pinlabel $1$ [b] at 264 105
\pinlabel $7$ [b] at 299 105
\pinlabel $6$ [b] at 308 105
\pinlabel $5$ [b] at 316 105
\pinlabel $\Omega$ [b] at 70 53
\pinlabel $\Omega$ [b] at 255 53
\pinlabel $1$ [b] at 145 45
\pinlabel $7$ [b] at 179 45
\pinlabel $1$ [r] at 9 23
\pinlabel $7$ [r] at 43 23
\pinlabel $1$ [r] at 53 23
\pinlabel $7$ [r] at 86 23
\pinlabel $1$ [r] at 95 23
\pinlabel $7$ [r] at 128 23
\pinlabel $1$ [r] at 197 23
\pinlabel $7$ [r] at 231 23
\pinlabel $1$ [r] at 240 23
\pinlabel $7$ [r] at 273 23
\pinlabel $1$ [r] at 283 23
\pinlabel $7$ [r] at 316 23
\pinlabel $1$ [b] at 332 45
\pinlabel $7$ [b] at 366 45
\endlabellist
\centering
\includegraphics[scale=0.75]{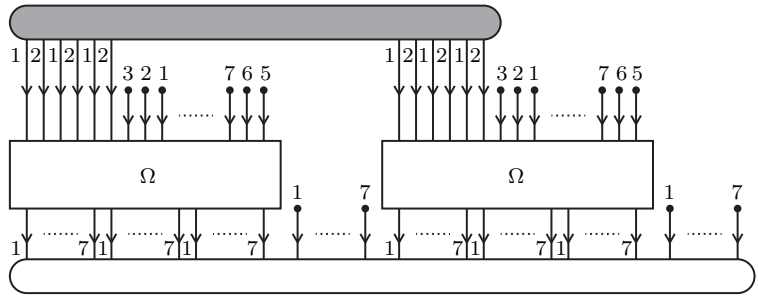}
%\bigskip
\caption{Chart $\Gamma_V$}
\label{chartN}
\end{figure}

\begin{figure}[ht!]
\labellist
\footnotesize \hair 2pt
\pinlabel $\Omega$ [b] at 68 163
\pinlabel $=$ [b] at 160 165
\pinlabel $1$ [b] at 193 341
\pinlabel $2$ [b] at 205 341
\pinlabel $1$ [b] at 216 341
\pinlabel $2$ [b] at 227 341
\pinlabel $1$ [b] at 238 341
\pinlabel $2$ [b] at 250 341
\pinlabel $3$ [b] at 261 341
\pinlabel $2$ [b] at 272 341
\pinlabel $1$ [b] at 283 341
\pinlabel $4$ [b] at 294 341
\pinlabel $3$ [b] at 306 341
\pinlabel $2$ [b] at 317 341
\pinlabel $5$ [b] at 329 341
\pinlabel $4$ [b] at 340 341
\pinlabel $3$ [b] at 352 341
\pinlabel $6$ [b] at 363 341
\pinlabel $5$ [b] at 374 341
\pinlabel $4$ [b] at 385 341
\pinlabel $7$ [b] at 397 341
\pinlabel $6$ [b] at 409 341
\pinlabel $5$ [b] at 420 341
\pinlabel $1$ [t] at 193 0
\pinlabel $2$ [t] at 205 0
\pinlabel $3$ [t] at 216 0
\pinlabel $4$ [t] at 227 0
\pinlabel $5$ [t] at 238 0
\pinlabel $6$ [t] at 250 0
\pinlabel $7$ [t] at 261 0
\pinlabel $1$ [t] at 272 0
\pinlabel $2$ [t] at 283 0
\pinlabel $3$ [t] at 294 0
\pinlabel $4$ [t] at 306 0
\pinlabel $5$ [t] at 317 0
\pinlabel $6$ [t] at 329 0
\pinlabel $7$ [t] at 340 0
\pinlabel $1$ [t] at 352 0
\pinlabel $2$ [t] at 363 0
\pinlabel $3$ [t] at 374 0
\pinlabel $4$ [t] at 385 0
\pinlabel $5$ [t] at 397 0
\pinlabel $6$ [t] at 409 0
\pinlabel $7$ [t] at 420 0
\pinlabel $2$ [r] at 260 302
\pinlabel $3$ [r] at 294 268
\pinlabel $4$ [r] at 328 233
\pinlabel $5$ [r] at 362 199
\pinlabel $6$ [r] at 396 166
\pinlabel $3$ [r] at 238 195
\pinlabel $2$ [r] at 249 195
\pinlabel $1$ [r] at 261 195
\pinlabel $4$ [r] at 272 195
\pinlabel $3$ [r] at 284 184
\pinlabel $2$ [r] at 295 184
\pinlabel $5$ [r] at 306 184
\pinlabel $4$ [r] at 318 173
\pinlabel $3$ [r] at 329 173
\pinlabel $6$ [r] at 340 173
\pinlabel $5$ [r] at 352 162
\pinlabel $4$ [r] at 363 162
\pinlabel $7$ [r] at 375 162
\pinlabel $3$ [tr] at 286 291
\pinlabel $4$ [tr] at 319 258
\pinlabel $5$ [tr] at 354 224
\pinlabel $6$ [tr] at 388 190
\pinlabel $2$ [tr] at 286 257
\pinlabel $3$ [tr] at 320 223
\pinlabel $4$ [tr] at 354 190
\pinlabel $5$ [tr] at 388 156
\pinlabel $3$ [tr] at 250 157
\pinlabel $4$ [tr] at 273 134
\pinlabel $5$ [tr] at 295 113
\pinlabel $6$ [tr] at 319 89
\pinlabel $6$ [tl] at 337 88
\endlabellist
\centering
\includegraphics[scale=0.75]{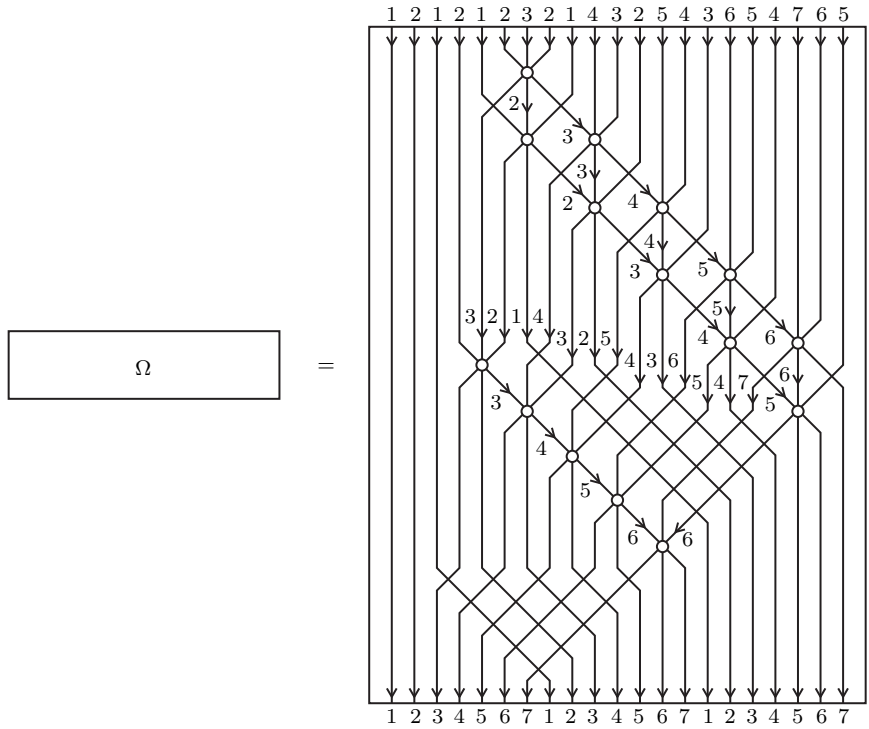}
\bigskip
\caption{Inside of the box labeled $\Omega$}
\label{chartL}
\end{figure}

Since both $\Gamma_U$ and $\Gamma_V$ have a white vertex of type $\hat{r}_4$, 
we conclude that $w(f_U)=w(f_V)=1$, 
and that $f_U$ and $f_V$ are stably isomorphic by Theorem \ref{stable}. 
In fact it is not difficult to see that the (untwisted) $\mathcal{H}$--fiber sums 
$f_U\# f_0$ and $f_V\# f_0$ are 
$\mathcal{H}$--isomorphic (and hence isomorphic by Proposition \ref{image}) 
by applying chart moves to $\Gamma_U\oplus \Gamma_0$ and 
$\Gamma_V\oplus \Gamma_0$. 
It is not clear whether $f_U$ and $f_V$ are $\mathcal{H}$--isomorphic, 
and whether $M_U$ and $M_V$ are diffeomorphic. 
It would be worth noting that $M_U$ and $M_V$ can not be distinguished 
by Usher's theorem \cite{Usher2006} 
because both $f_U$ and $f_V$ have a $(-1)$--section. 
Two Hurwitz systems $U$ and $V$ are related by a `partial twisting' operation, 
namely, $V$ is obtained from $U$ by replacing $(\zeta_1,\ldots ,\zeta_7)$ with 
$(\zeta''_1,\ldots ,\zeta''_7)$. 
Such pairs of Hurwitz systems often yield pairs of $4$--manifolds with subtle difference 
(see works of Auroux \cite{Auroux2006} and Yasui \cite{Yasui2014}). 
\end{rem}

%%%%%%%%%%%%%%%%%%%%%%%%%%

\subsection*{Acknowledgements}\label{ackref}

%%%%%%%%%%%%%%%%%%%%%%%%%%

The first author would like to thank Kokoro Tanaka and Isao Hasegawa 
for helpful discussions on charts and central extensions. 
The authors would like to thank the referee for helpful suggestions and comments. 
The first author was partially supported by JSPS KAKENHI Grant Numbers 
21540079, 25400082, 16K05142. 
The second author was partially supported by JSPS KAKENHI Grant Numbers 
21340015, 26287013.

\end{document}